\documentclass[yjsco, preprint]{elsarticle}
\usepackage{amsmath}\usepackage{amssymb}
\usepackage{amsthm}
\usepackage[english]{babel}
\usepackage{float}
\usepackage{graphics}
\usepackage{graphicx}
\usepackage[thinlines,thiklines]{easybmat}
\usepackage{multirow}
\usepackage{algorithmic}
\usepackage{algorithm}
\usepackage{amscd}

\newcommand{\Z}{\mathbb{Z}}
\newcommand{\Q}{\mathbb{Q}}
\newcommand{\R}{\mathbb{R}}
\newcommand{\C}{\mathbb{C}}
\newcommand{\F}{\mathbb{F}}

\newcommand{\OK}{\mathcal{O}_{K}}

\newcommand{\ag}{\mathfrak{a}}
\newcommand{\bg}{\mathfrak{b}}

\newcommand{\dg}{\mathfrak{d}}
\newcommand{\g}{\mathfrak{g}}

\newcommand{\OT}{\tilde O}
\newcommand{\bs}{\mathsf S}

\let\qedhere\relax

\newcommand{\Bid}{B_\mathrm{id}}
\newcommand{\Be}{B_\mathrm{e}}

\providecommand{\inorm}[1]{\operatorname{N}(#1)}
\providecommand{\HNF}{\operatorname{HNF}}
\providecommand{\HOW}{\operatorname{How}}

\newtheorem{theorem}{Theorem}
\newtheorem*{theorem-non}{Theorem}

\newtheorem{lemma}[theorem]{Lemma}

\newtheorem{proposition}[theorem]{Proposition}
\newtheorem{corollary}[theorem]{Corollary}

\newtheorem{remark}[theorem]{Remark}

\begin{document}
\begin{frontmatter}
\title{On the computation of the HNF of a module over the ring of integers of a number field}    
\author{Jean-Fran\c{c}ois Biasse}
\address{Department of Mathematics and Statistics\\University of South Florida}
\ead{biasse@usf.edu}

\author{Claus Fieker}
\address{Fachbereich Mathematik\\Universit\"{a}t Kaiserslautern\\Postfach 3049\\67653 Kaiserslautern - Germany}
\ead{fieker@mathematik.uni-kl.de}

\author{Tommy Hofmann}
\address{Fachbereich Mathematik\\Universit\"{a}t Kaiserslautern\\Postfach 3049\\67653 Kaiserslautern - Germany}
\ead{thofmann@mathematik.uni-kl.de}

\begin{keyword}
Number field\sep  Dedekind domain\sep  Hermite normal form\sep  relative extension of number fields
\MSC 11Y40
\end{keyword}

%%%%%%%%%%%%%%%%%%%%%%%%%%%%%%%%%%%%%%%%%%%%%%%%%%%%%%%%%%%%%%%%%%%%%%%%%%%%%%%%
%
%  Abstract
%
%%%%%%%%%%%%%%%%%%%%%%%%%%%%%%%%%%%%%%%%%%%%%%%%%%%%%%%%%%%%%%%%%%%%%%%%%%%%%%%%

\begin{abstract}
We present a variation of the modular algorithm for computing the Hermite normal form of an $\OK$-module presented by 
Cohen~\cite{Cohen1996}, where $\OK$ is the ring of integers of a number field $K$. An approach presented in~\cite{Cohen1996} 
based on reductions modulo ideals was 
conjectured to run in polynomial time by Cohen, but so far, 
no such proof was available in the literature. In this paper, we present a modification of the approach of~\cite{Cohen1996} to 
prevent the coefficient swell and we rigorously assess its complexity with respect to the size of the input 
and the invariants of the field $K$.
\end{abstract}

\end{frontmatter}

\section{Introduction}

Algorithms for modules over the rational integers such as the Hermite normal form algorithm
are at the core of all methods for computations with rings and ideals in 
finite extensions of the rational numbers. Following the growing
interest in relative extensions, that is, finite extensions of number fields, 
the structure of modules over Dedekind domains became 
important. On the theoretical side, it was well known that the framework
of finitely generated projective modules was well suited for these problems, but explicit algorithms
were lacking for a long time. Based on the pioneering work of
of Bosma and Pohst~\cite{Bosma1991}, the computation of a Hermite normal form (HNF) over principal ideal domains was generalized to finitely generated modules over Dedekind domains by Cohen~\cite{Cohen1996} (for a comparison between the work of Bosma--Pohst and Cohen, see~\cite[Chap. 6]{Hoppe1998}). 
It was conjectured that Cohen's algorithm \cite[Algorithm 3.2]{Cohen1996} for computing this so-called pseudo-Hermite normal form (pseudo-HNF) has polynomial complexity (see \cite[Remark after Algorithm 3.2]{Cohen1996}): ``\textellipsis and it seems plausible that \ldots this algorithm is, in fact, polynomial-time.''
The polynomial complexity of a (modified) version of Cohen's algorithm was conjectured in the folklore
but not formally proved until the preliminary version of this study in
the \mbox{ISSAC} proceedings \cite{issac12}.
The difficulties in establishing this formally were two-fold:
The original algorithm does not control the size of the coefficient ideals, and,
most of the underlying field and ideal operations themselves have not been
analyzed completely.
While the ideal operations, which are based on Hermite normal forms over the rational integers,
are known to have polynomial complexity, the exact complexity was previously
not investigated in detail hence
a byproduct of this discussion is a computational model for algebraic number fields together with an analysis of basic field and ideal operations.

Based on our careful analysis we also compare the complexity of
algorithms for finitely generated projective modules over the ring of integers $\OK$ of a number field $K$ based on the 
structure as $\OK$-modules with algorithms based on the structure as free
$\Z$-modules of larger rank.
In practice, algebraic number fields $L$ of large degree are carefully constructed as relative extensions $\Q \subseteq K \subseteq L$.
The computational complexity of element and ideal operations in $L$ depend on both $d = [K:\Q]$ and $n = [L:K]$.
Ideals of the ring of integers $\mathcal O_L$ of $L$ are naturally $\Z$-modules of rank $dn$ and therefore ideal arithmetic is reduced to computation of $\Z$-modules of rank $dn$.
On the other hand, the ring of integers $\mathcal O_L$ and its ideals are finitely generated projective modules of rank $n$ over the Dedekind domain $\mathcal O_K$.
Thus the ideal arithmetic in $\mathcal O_L$ can be performed using the pseudo-HNF algorithm and it is only natural then to ask which method to prefer.
\par
In addition, Fieker and Stehl{\'e}'s recent algorithm for computing a reduced basis of $\mathcal O_K$-modules relies on the conjectured possibility to compute a pseudo-HNF for an $\OK$-module with polynomial complexity~\cite[Th. 1]{Fieker2010}.
This allows a reduction algorithm for $\OK$-modules with polynomial complexity, similar to the LLL algorithm for $\Z$-modules.

In the same way as for $\Z$-modules, where the HNF can be used to
compute the Smith normal form, the pseudo-HNF enables us to determine a pseudo-Smith
normal form.
The pseudo-Smith normal form gives the structure of torsion $\OK$-modules, and
is used to study the quotient of two modules. Applications include the
investigation of Galois cohomology \cite{McQuillian}.

In all of our algorithms and the analysis we assume that the maximal order, $\OK$, is part
of the input.

\subsection*{Our contribution} 
Let $K$ be a number field with ring of integers $\OK$.
We present in this paper the first algorithm for computing a
pseudo-HNF of an $\OK$-module which has a proven polynomial complexity. Our algorithm is based on the modular approach of Cohen~\cite[Chap.
1]{Cohen2000} extending and correcting the version from the \mbox{ISSAC} proceedings \cite{issac12}.
We derive bounds on its complexity with respect to the size of
the input, the rank of the module and the invariants of the field.

As every $\OK$-module is naturally a $\Z$-module (of larger rank), we then compare the complexity of module operations as $\OK$-modules
to the complexity of the same operations as $\Z$-modules.
In particular, we show that the complexity of the $\OK$-module approach 
with respect to the degree of the field $K$ is (much) worse
than in the $\Z$-module approach. This is due to the (bad) performance of our
key tool: An algorithm to establish tight bounds on the norms of the
coefficient ideals during the pseudo-HNF algorithm.

As an application of our algorithm, we extend the techniques to also give
an algorithm with polynomial complexity to compute the pseudo-Smith normal form 
associated to $\OK$-modules, which is a constructive variant of the elementary divisor theorem for modules over $\OK$. Similarly to the pseudo-HNF, this is the first algorithm
for this task that is proven to have polynomial complexity.

\subsection*{Outline}
In order to discuss the complexity of our algorithms, we start by introducing
our computational model and natural representations of the involved objects.
Next, suitable definitions for size of the objects are introduced and the behavior under necessary operations is analyzed.

Once the size of the objects is settled, we proceed to develop algorithms
for all basic operations we will encounter and prove complexity
results for all algorithms. In particular, this section contains
algorithms and their complexity for most common ideal operations in
number fields. While most of the methods are folklore, this is the
first time their complexity has been stated.

Next, the key new technique, the normalization of the coefficient ideals,
is introduced and analyzed. Finally, after all the tools are in place,
we move to the module theory.
Similar to other modular algorithms, we first need to find a suitable modulus.
Here this is the determinantal ideal, which is the product of fractional ideals
and the determinant of a matrix with entries in $\OK$.
In Section 5 we 
present a Chinese remainder theorem based algorithm for the determinant
computation over rings of integers and analyze its complexity.

In Section 6, we get to the main result: An explicit algorithm
that will compute a pseudo-HNF for any full rank module
over the ring of integers. The module is specified via a pseudo-generating
system (pairs of fractional ideals of the number field $K$ and vectors in
$K^m$). Under the assumption that the module has full rank and that it is
contained in $\OK^m$, we prove the following (see Theorem \ref{thm:pseudohnf}):

\begin{theorem-non}
  There exists an algorithm (Algorithm \ref{alg:pseudohnf}), that given $n$
  pseudo-generators of an $\OK$-module of full rank contained in $\OK^m$,
  computes a pseudo-HNF with polynomial complexity.
\end{theorem-non}

Actually, a more precise version is proven. The exact dependency on the 
ring of integers $\OK$, the dimension of the module and the size of the
generators is presented.
Note that we assume that certain data of the number field $K$ is precomputed,
including an integral basis of the ring of integers (see
Section~\ref{sec:sizecost}).

In the final section, we apply the pseudo-HNF algorithm to derive a
pseudo-Smith normal form algorithm and analyze its complexity,
achieving polynomial time complexity as well (Algorithm~\ref{alg:pseudosnf} and
Proposition~\ref{prop:pseudosnf}).

%%%%%%%%%%%%%%%%%%%%%%%%%%%%%%%%%%%%%%%%%%%%%%%%%%%%%%%%%%%%%%%%%%%%%%%%%%%%%%%%
%
%  Preliminaries
%
%%%%%%%%%%%%%%%%%%%%%%%%%%%%%%%%%%%%%%%%%%%%%%%%%%%%%%%%%%%%%%%%%%%%%%%%%%%%%%%%

\section{Preliminaries}

\subsection*{Number fields}
Let $K$ be a number field of degree $d$ and signature $(r_1,r_2)$.
That is $K$ admits $r_1$ real embeddings and $2r_2$ complex embeddings. We can embed $K$ in $K_\R = K \otimes_\Q \R \simeq \R^{r_1} \times \C^{r_2}$ and extend all embeddings to $K_\R$.
The $d$-dimensional real vector space $K_\R$ carries a Hermitian form $T_2$ defined by $T_2(\alpha,\beta) = \sum_{\sigma} \sigma(\alpha) \overline \sigma(\beta)$ for $\alpha,\beta \in K_\R$, where the sum runs over all embeddings, and an associated norm $\left\| \phantom{\alpha} \right\|$ defined by $\left\| \alpha \right\| = \sqrt{T_2(\alpha,\alpha)}$ for $\alpha \in K_\R$.
The ring $\mathcal O_K$ of algebraic integers is the maximal order of $K$ and therefore a $\Z$-lattice of rank $d$ with $\mathcal O_K \otimes_\Z \Q = K$.
Given any $\Z$-basis $\omega_1,\dotsc,\omega_d$ of $\mathcal O_K$, the discriminant $\Delta_K$ of the number field $K$ is defined as $\Delta_K = \det(\operatorname{Tr}(\omega_i\omega_j)_{i,j})$, where $\operatorname{Tr}$ denotes the trace of the finite field extension $\mathbb Q \subseteq K$.
The norm of an element $\alpha \in K$ is defined by $\inorm{\alpha} = \operatorname{N}^K_\Q(\alpha) = \prod_\sigma \sigma(\alpha)$ and is equal to the usual field norm of the algebraic extension $K\supseteq \Q$.
For $\alpha \in K$, $M_\alpha$ denotes the $d \times d$ rational matrix
corresponding to $K \to K, \beta \mapsto \alpha \beta$, with respect to a
$\mathbb Q$-basis of $K$ and is called the regular representation of $\alpha$.
Here, using a fixed $\mathbb Q$-basis of $K$, elements are identified with
row-vectors in $\mathbb Q^{1\times d}$.
\par
To represent $\OK$-modules we rely on a generalization of the notion of ideal, namely the fractional ideals of $\OK$.
They are defined as finitely generated $\Z$-submodules of $K$.
When a fractional ideal is contained in $\OK$, we refer to it as an integral ideal, which is in fact an ideal of the ring $\OK$.
Otherwise, for every fractional ideal $\mathfrak a$ of $\OK$, there exists $r\in\Z_{>0}$ such that $r\mathfrak a$ is integral. The minimal positive integer with this property is defined as the denominator of the fractional ideal $\mathfrak a$ and is denoted by $\operatorname{den}(\mathfrak a)$.
The sum of two fractional ideals of $\OK$ is the usual sum as $\Z$-modules and the product of two fractional ideals $\mathfrak a$, $\mathfrak b$ is given by the $\Z$-module generated by $\alpha \beta$ with $\alpha \in \mathfrak a$ and $\beta \in \mathfrak b$.
The set of fractional ideals of $\OK$ forms a monoid with identity $\OK$ and where the inverse of $\mathfrak a$ is  
$\mathfrak a^{-1}:= \{ \alpha \in K\mid \alpha \mathfrak a\subseteq \OK\}$.
Each fractional ideal $\mathfrak a$ of $K$ is a free $\Z$-module of rank $d$ and given any $\Z$-basis matrix $N_\mathfrak a \in \Q^{d\times d}$ we define the norm $\operatorname{N}(\mathfrak a)$ of $\mathfrak a$ to be $\lvert \det(N_\mathfrak a)\rvert \in \Q$.
%The set of fractional ideals is equipped with a norm function defined by $\inorm{\mathfrak a} = \lvert \det(N_\mathfrak a) \rvert$ where $N_\mathfrak a$ is any $\Z$-basis matrix of $\mathfrak a$.
The norm is multiplicative, and in the case $\mathfrak a$ is an integral ideal the norm of $\mathfrak a$ is equal to $[\OK : \mathfrak a]$, the index of $\mathfrak a$ in $\mathcal O_K$.
Also note that the absolute value of the norm of $\alpha \in K$ agrees with the norm of the principal ideal $\alpha \OK$.

\subsection*{$\mathcal O_K$-modules and the pseudo-Hermite normal form over Dedekind domains}
In order to describe the structure of modules over Dedekind domains we rely on the notion of pseudoness introduced by Cohen \cite{Cohen1996}, see also \cite[Chapter 1]{Cohen2000}. Note that, different to \cite{Cohen1996}, our modules are generated by row vectors instead of column vectors and we therefore perform row operations. Let $M$ be a non-zero finitely generated torsion-free $\mathcal O_K$-module and $V = K\otimes_{\mathcal O_K} M$, a finite dimensional $K$-vector space containing $M$.
An indexed family $(\alpha_i,\mathfrak a_i)_{1 \leq i \leq n}$ consisting of $\alpha_i \in V$ and fractional ideals $\mathfrak a_i$ of $K$ is called a pseudo-generating system of $M$ if
\[ M = \mathfrak a_1 \alpha_1 + \dotsb + \mathfrak a_n \alpha_n \]
and a pseudo-basis of $M$ if
\[ M = \mathfrak a_1 \alpha_1 \oplus \dotsb \oplus \mathfrak a_n \alpha_n. \]
A pair $(A,I)$ consisting of a matrix $A \in K^{n\times m}$ and a list of fractional ideals $I = (\mathfrak a_i)_{1 \leq i \leq n}$ is called a pseudo-matrix.
Denoting by $A_1,\dotsc,A_n \in K^{1 \times m}$ the $n$ rows of $A$, the sum $\sum_{i=1}^n \mathfrak a_i A_i$ is a finitely generated torsion-free $\mathcal O_K$-module associated to this pseudo-matrix.
Conversely every finitely generated torsion-free module $M$ gives rise to a pseudo-matrix whose associated module is $M$.
In case of finitely generated torsion-free modules over principal ideal domains, the task of finding a basis of the module can be reduced to finding the Hermite normal form (HNF) of the associated matrix.
If the base ring is a Dedekind domain there exists a canonical form for pseudo-matrices, the pseudo-Hermite normal form (pseudo-HNF), which plays the same role as the HNF for principal ideal domains allowing us to construct pseudo-bases from pseudo-generating systems.
More precisely let $A \in K^{n \times m}$ be of rank $m$, $(A,I)$ a pseudo-matrix and $M$ the associated $\mathcal O_K$-module. Then there exists an $n\times n$ matrix $U = (u_{i,j})_{i,j}$ over $K$ and $n$ non-zero fractional ideals $\mathfrak b_1,\dotsc,\mathfrak b_n$ of $K$ satisfying
\begin{enumerate}
\item
  for all $1 \leq i,j \leq n$ we have $u_{i,j} \in \mathfrak b_i^{-1} \mathfrak a_j$,
\item
  the ideals satisfy $\prod_i \mathfrak a_i = \det(U) \prod_i \mathfrak b_i$,
\item
  the matrix $UA$ is of the form
  \[ \left( \begin{array}{c} H \\ \hline \mathbf{0} \end{array}\right),\]
  where $H$ is an $m \times m$ lower triangular matrix over $K$ with $1$'s on the diagonal and $\mathbf{0}$ denotes the zero matrix of suitable dimensions.
\item
  $M = \mathfrak b_1 H_1 \oplus \dotsb \oplus \mathfrak b_n H_n$ where $H_1,\dotsc,H_n$ are the rows of $H$.
\end{enumerate}
\par 
The pseudo-matrix $(H,(\mathfrak b_i)_{1 \leq i \leq n})$ is called a pseudo-Hermite normal form (pseudo-HNF) of $(A,I)$ resp. of $M$. Note that with this definition, a pseudo-HNF of an $\OK$-module is not unique. In~\cite{Bosma1991,Cohen1996,Hoppe1998}, reductions of the coefficients of $A$ modulo certain ideals provide uniqueness of the pseudo-HNF when the reduction algorithm is fixed.
\par
Throughout the paper will make the following restriction: We assume that the associated module $M$ is a subset of $\mathcal O_K^m$. For if $M \subseteq K^m$ there exists an integer $k \in \Z_{>0}$ such that $kM \subseteq \mathcal O_K^m$.
In case of a square pseudo-matrix $(A, I)$ the determinantal ideal $\mathfrak d((A, I))$ is defined as to be $\det(A)\prod_{\mathfrak a \in I} \mathfrak a$.
For a pseudo-matrix $(A, I)$, $A \in \mathcal O_K^{n \times m}$ of rank $m$, we define the determinantal ideal $\mathfrak d((A, I))$ to be the $\gcd$ of all determinantal ideals of all $m \times m$ sub-pseudo-matrices of $(A, I)$ (see \cite{Cohen1996}).

%%%%%%%%%%%%%%%%%%%%%%%%%%%%%%%%%%%%%%%%%%%%%%%%%%%%%%%%%%%%%%%%%%%%%%%%%%%%%%%%
%
%  Size and Costs in Algebraic Number Fields
%
%%%%%%%%%%%%%%%%%%%%%%%%%%%%%%%%%%%%%%%%%%%%%%%%%%%%%%%%%%%%%%%%%%%%%%%%%%%%%%%%

\section{Size and costs in algebraic number fields}\label{sec:sizecost}

In order to state the complexity of the pseudo-HNF algorithm, we will now describe representations and algorithms of elements and ideals in number fields, which are the objects we have to compute with.
The algorithms and representations chosen here are by no means optimal for all problems involving algebraic number fields.
We have chosen the linear algebra heavy approach since it allows for efficient algorithms of the normalization of ideals and reduction of elements with respect to ideals, which are crucial steps in the pseudo-HNF algorithm.
For different approaches to element arithmetic we refer the interested reader to~\cite[4.2]{Cohen1993} and~\cite{Belabas2004}.
For ideal arithmetic (in particular ideal multiplication) fast Las Vegas type algorithm are available making use of a 2-element ideal representation (see \cite{Cohen1993, Belabas2004}).
As our aim is a \textit{deterministic} polynomial time pseudo-HNF algorithm, we will not make use of them.

\subsection*{A notion of size.} 
To ensure that our algorithm for computing a pseudo-HNF basis of an $\OK$-module runs in polynomial time, we need a notion of size that bounds the bit size required to represent ideals and field elements.
We assume that the maximal order $\mathcal O_K$ of $K$ is given by a fixed $\Z$-basis $\Omega = (\omega_1,\dotsc,\omega_d)$ with $\omega_1 = 1$.

\subsubsection*{Size of ideals}
A non-zero integral ideal $\mathfrak a \subseteq \OK$ is a $d$-dimensional $\Z$-submodule of $\OK$ and will be represented by its unique (lower triangular) HNF basis $M_\mathfrak a \in \Z^{d\times d}$ with respect to the fixed integral basis $\Omega$.
The size required to store the matrix is therefore bounded by $d^2\log(\lvert M_\mathfrak a \rvert)$, where $\log$ denotes the binary logarithm.
Since we assume that $\omega_1$ is set to $1$ the value $\lvert M_\mathfrak a \vert$ is actually equal to $\min\{ a \in \Z_{>0} \, | \, a \in \mathfrak a \}$.
(For $A = (a_{i,j})_{i,j} \in \Z^{d \times d}$ we denote $\max_{i,j} \lvert a_{i,j} \rvert$ by $\lvert A \rvert$.)
The latter is the well known \textit{minimum} of the integral ideal $\mathfrak a$, which is denoted by $\min(\mathfrak a)$ and can be characterized as the unique positive integer with $\mathfrak a \cap \Z = (\min(\mathfrak a))$.
Based on this observation we define
\[ \bs(\mathfrak a) = d^2 \log(\min(\mathfrak a)) \]
to be the \textit{size} of $\mathfrak a$.
If $\mathfrak a = {\tilde{\mathfrak a}}/{k}$ is a fractional ideal of $K$, where $\mathfrak a\subseteq\OK$ and $k \in\Z_{>0}$ is the denominator of $\mathfrak a$, we define the \textit{size} of $\mathfrak a$ by 
\[ \bs(\mathfrak a) = \bs(\tilde{\mathfrak a}) + d^2 \log(k). \]
The weight $d^2$ on the denominator is introduced to have a nice behavior 
with respect to the common ideal operations. Before we show that,  we need to recall some basic facts about the minimum of integral ideals.
The weight can also be seen as viewing the ideal as given by a rational matrix directly.

\begin{proposition}\label{size:minimum}
  Let $\mathfrak a, \mathfrak b$ be integral ideals and $k \in \Z$, $k \neq 0$. Then the following holds:
  \begin{enumerate}
    \item
      $\min(\mathfrak a + \mathfrak b)$ divides $\operatorname{GCD}(\min(\mathfrak a),\min(\mathfrak b))$.
    \item
      $\min(\mathfrak a \mathfrak b)$ divides $\min(\mathfrak a)\min(\mathfrak b)$.
    \item
      The denominator of $\mathfrak a^{-1}$ is equal to $\min(\mathfrak a)$.
    \item
      $\min(k\mathfrak a) = \lvert k\rvert \min(\mathfrak a)$.
    \item
      $\min(\mathfrak a)$ divides $\inorm {\mathfrak a}$.
  \end{enumerate}
\end{proposition}

\begin{proof}
  Follows from the definition.
\end{proof}

The properties of the minimum translate easily into corresponding properties of the size of integral ideals.
The next proposition shows that in fact the same relations hold also for fractional ideals.

\begin{proposition}\label{size:ideals}
  Let $\mathfrak a,\mathfrak b \subseteq K$ be fractional ideals and $m \in \Z$, $m \neq 0$. Then the following holds:
  \begin{enumerate}
    \item $\bs(m\mathfrak a) \leq \bs(\mathfrak a) + d^2 \log(\lvert m \rvert)$.
    \item $\bs(\mathfrak a + \mathfrak b) \leq 2(\bs(\mathfrak a)+\bs(\mathfrak b))$.
    \item $\bs(\mathfrak a \mathfrak b) \leq \bs(\mathfrak a) + \bs(\mathfrak b)$
    \item $\bs(\mathfrak a ^{-1}) \leq 2\bs(\mathfrak a)$.
  \end{enumerate}
\end{proposition}

\begin{proof}
  Note that if $\mathfrak a$ and $\mathfrak b$ are integral ideals then (1), (2) and (3) follow immediately from the properties of the minimum obtained in Proposition~\ref{size:minimum}.
  Write $\mathfrak a = {\tilde{\mathfrak a}}/{k}$ and $\mathfrak b = {\tilde{\mathfrak b}}/{l}$ with $k$ and $l$ the denominator of $\mathfrak a$ and $\mathfrak b$ respectively.\\
  (1): We have
  \[ \bs(m \mathfrak a) \leq \bs(m\tilde{\mathfrak a}) + d^2 \log(k) \leq \bs(\tilde{\mathfrak a}) + d^2 \log(\lvert m \rvert) + d^2 \log(k) = \bs(\mathfrak a) + d^2 \log(\lvert m \rvert). \]
  (2): As the sum $\mathfrak a + \mathfrak b$ is equal to $(l\tilde{\mathfrak a} + k \tilde{\mathfrak b})/{kl}$ we obtain
  \begin{align*} \bs(\mathfrak a + \mathfrak b) \leq \bs(l\tilde{\mathfrak a} + k \tilde{\mathfrak b}) + d^2\log(kl) &\leq \bs(\tilde{\mathfrak a}) + d^2 \log(l) + \bs(\tilde{\mathfrak b}) + d^2 \log(k) + \bs(\mathfrak a) + \bs(\mathfrak b) \\
    &= 2 ( \bs(\mathfrak a) + \bs(\mathfrak b)). \end{align*}
  (3): We have
  \[ \bs(\mathfrak a \mathfrak b) \leq \bs(\tilde{\mathfrak a}\tilde{\mathfrak b}) + d^2 \log(k) + d^2 \log(l) \leq \bs(\mathfrak a) + \bs(\mathfrak b). \]
  (4): Consider first the integral case: We know that $\min(\tilde{\mathfrak a}) \in \tilde{\mathfrak a}$.
  Thus the principal ideal $(\min(\tilde{\mathfrak a}))$ is divided by $\tilde{\mathfrak a}$ and there exists an integral ideal $\mathfrak b$ with $(\min(\tilde{\mathfrak a})) = \tilde{\mathfrak a} \mathfrak b$, i.\,e.,
  \[ \tilde{\mathfrak a}^{-1} = \frac{\mathfrak b}{\min(\tilde{\mathfrak a})}. \]
  Note that $\min(\tilde{\mathfrak a}) \in \mathfrak b$ and therefore $\min(\mathfrak b) \leq \min(\tilde{\mathfrak a})$. 
  As $\min(\tilde{\mathfrak a})$ is the denominator of $\tilde{\mathfrak a}^{-1}$ by Proposition~\ref{size:minimum}~(4) we obtain
  \[ \bs(\tilde{\mathfrak a}^{-1}) = \bs(\mathfrak b) + d^2 \log(\min(\tilde{\mathfrak a})) \leq 2 \bs(\tilde{\mathfrak a}). \]
  Returning to the general case we have $\mathfrak a^{-1} = k \tilde{\mathfrak a}^{-1}$.
  Then
  \[ \bs(\mathfrak a^{-1}) = \bs(\tilde{\mathfrak a}^{-1}) + d^2 \log(k) \leq 2 \bs(\tilde{\mathfrak a}) + 2 d^2 \log(k) = 2 \bs(\mathfrak a). \qedhere\]
\end{proof}

\subsubsection*{Size of elements.}
The integral basis $\Omega$ allows us to represent an integral element $\alpha \in \mathcal O_K$ by its coefficient vector $(a_1,\dotsc,a_d) \in \Z^d$ satisfying $\alpha = \sum_{i=1}^d a_i \alpha_i$.
The size to store the element $\alpha$ is therefore bounded by
\[ \bs(\alpha) = d \max_{i} \log( |a_i|),\]
which we call the \textit{size} of $\alpha$ with respect to $\Omega$.
This can be faithfully generalized to elements $\alpha \in K$.
Writing $\alpha= \tilde\alpha/k$ with $k \in \Z_{>0}$ the denominator of $\alpha$ we define
\[ \bs(\alpha) = \bs(\tilde\alpha) + d\log(k) \]
to be the size of $\alpha$. Similarly to  the ideals above, as added the weight
$d$ to the denominator to achieve a nicer transformation behavior under
the standard operations. Its justification also comes from viewing elements
in $K$ as rational vectors rather than integral elements with a common
denominator.
\par
In order to relate our function $\bs$ to the multiplicative structure on $K$ we need to recall that the notion of size of elements is closely related to norms on the $\R$-vector space $K_\R$.
More precisely, the fixed integral basis $\Omega$ gives rise to an isomorphism\[ \Phi \colon K_\R \to \R^d, \, \sum_{i=1}^d a_i \omega_i \longmapsto (a_1,\dotsc,a_d), \]
onto the $d$-dimensional real vector space.
Equipping $\R^d$ with the $\infty$-norm we have $d \log (\| \Phi(\alpha) \|_\infty) = \bs(\alpha)$ for $\alpha \in \OK$.
But this is not the only way to identify $K_\R$ with a normed real vector space.
Denote the $r_1$ real embeddings by $(\sigma_i)_{1 \leq i \leq r_1}$ and the $2r_2$ complex embeddings by $(\sigma_i)_{r_1 + 1 \leq i \leq r_1 +2r_2}$.
We use the usual ordering of the complex embeddings, such that $\sigma_{r_2+k} = \overline \sigma_{k}$ for $r_1 < k \leq r_1 + r_2$. Using these embeddings we define
\begin{align*} \Psi \colon K_\R &\longrightarrow  \R^{r_1} \times \R^{2r_2} \\
  \alpha &\longmapsto (\sigma_i(\alpha))_{1 \leq i \leq r},(\operatorname{Re} \sigma_i(\alpha) + \operatorname{Im} \sigma_i(\alpha),\operatorname{Re} \sigma_i(\alpha) - \operatorname{Im} \sigma_i(\alpha))_{r_1 < i \leq 2r_2 + 1}
\end{align*}
yielding $\left\|\Psi(\alpha)\right\|_2 = \|\alpha\|$ for $\alpha \in K$, where $\left\|\phantom{x}\right\|_2$ denotes the $2$-norm on $\mathbb R^{r_1+2r_2}$.
Since $\R$ is complete, any two norms on $K_\R$ are equivalent.
Thus there exists constants $C_1,C_2 \in \R_{>0}$ depending on $K$ and the chosen basis $\Omega$ with
\begin{equation}\label{ineq:1}
    \frac 1 {C_2} \left\| \Phi(\alpha) \right\|_\infty \leq \|\alpha\| \leq C_1 \|\Phi(\alpha)\|_\infty,
\end{equation}
for all $\alpha \in K$.
Moreover we have the inequalities
\begin{equation}\label{ineq:2}
  \|\alpha\| \leq \sqrt d \, \max_\sigma \left|\sigma(\alpha)\right|, \quad \max_\sigma \left|\sigma(\alpha)\right| \leq \|\alpha\|,
\end{equation}
for all $\alpha \in K$ and applying the geometric arithmetic mean inequality yields
\begin{equation}\label{ineq:3}
  \lvert \inorm{\alpha} \rvert \leq \frac{\left\| \alpha \right\|^d}{d^{d/2}}.
\end{equation}

Another important characteristic of an integral basis $\Omega$ is the size of the structure constants $(m_{i,j}^k)_{i,j,k}$, which are defined by the relations
\[ \omega_i \omega_j = \sum_{k=1}^d m_{i,j}^k \omega_k \]
for $1 \leq i,j \leq d$. We denote the maximum value $\max_{i,j,k} |m_{i,j}^k|$ by $C_3$.

\begin{remark} 
  Note that there is a situation in which we are able to estimate the constants $C_1,C_2,C_3$.
  Assume that $\Omega$ is LLL-reduced with respect to $T_2$ and LLL parameter $c$.
  Then by \cite[Proposition 5.1]{Belabas2004} the basis $\Omega$ satisfies
  \[ \left \| \omega_i \right \|^2 \leq \left( d^{-(i-1)} c^{d(d-1)/2} \lvert \Delta_K \rvert \right)^{1/(d-i+1)} \]
  for all $1 \leq i \leq d$. Moreover the structure constants satisfy
  \[ \lvert m_{i,j}^k\rvert \leq \frac{c^{3d(d-1)/4}}{d^{d-(1/2)}} \lvert \Delta_K \rvert \quad 1 \leq i,j,k \leq d, \]
  and thus we can choose
  \[ C_1 = \max_i \left( d^{-(i-1)} c^{d(d-1)/2} \lvert \Delta_K \rvert \right)^{1/2(d-i+1)}, \quad C_3 = \frac{c^{3d(d-1)/4}}{d^{d-(1/2)}}. \]
  By \cite[Lemma 2]{Fieker2010} we have  $\left\| \Phi(\alpha) \right\|_\infty \leq 2^{3d/2} \left\| \alpha \right\|$ for all $\alpha \in K$ allowing for $C_2 = 2^{3d/2}$.
\end{remark}

Using the preceding discussion we can now describe the relation between size and the multiplicative structure of $\mathcal O_K$.
If $\alpha = \sum_{i=1}^d a_i \omega_i$ and $\beta = \sum_{j=1}^d b_j \omega_j$ are integral elements the product $\alpha\beta$ is equal to $\sum_{k=1}^d c_k \omega_k$ with

\[ c_k = \sum_{i=1}^d a_i \sum_{j=1}^d b_j m_{i,j}^k .\]
Thus for the size of $\alpha\beta$ we obtain
\[ \bs(\alpha \beta) \leq \bs(\alpha) + \bs(\beta) + 2d \log(d) + d \log(C_3). \]
The constant $2d\log(d) + d\log(C_3)$ therefore measures the increase of size when multiplying two integral elements.
\par
The second multiplicative operation is the inversion of integral elements.
Let $\alpha^{-1} = \beta/k$ with $k \in \Z_{>0}$ the denominator of $\alpha^{-1}$ and $\beta \in \mathcal O_K$.
Using $k \leq \lvert\inorm{\alpha}\rvert$ and Inequality~(\ref{ineq:3}) we obtain $\log(k) \leq d \log (C_1) + d \log(\left\| \alpha \right\|_\infty) - \frac d 2 \log(d)$.
Since
\[ \lvert \sigma(\beta) \rvert = \frac{\sigma(k)}{\lvert \sigma(\alpha)\rvert} = \frac{k} {\lvert \sigma(\alpha) \rvert} \leq \frac{\lvert \inorm{\alpha} \rvert}{\lvert \sigma(\alpha) \rvert} = \prod_{\tau \neq \sigma} \lvert \tau(\alpha)\rvert \leq \left\| \alpha \right\|^{d-1} \]
for every embedding $\sigma$ we get $\left \| \beta \right\| \leq \sqrt d \left\| \alpha \right\|^{d-1}$ by Inequality~(\ref{ineq:2}).
Combining this with the estimate for the denominator yields
\[ \bs(\alpha^{-1}) = d \log(k) + \bs(\beta) \leq d \bs(\alpha) + d^2 \log(C_1) + d \log(C_2) .\]
Again we see that there is a constant depending on $\Omega$ describing the increase of size during element inversion.
We define $C_\Omega$ by
\[ C_\Omega = \max\{ 2d \log(d) + d\log(C_3), d^2 \log(C_1) + d\log(C_2) \} \]
to obtain a constant incorporating both operations.
Since we work with a fixed basis we drop the $\Omega$ from the index and denote this constant just by $C$.
So far the obtained bounds on the size are only valid for integral elements and it remains to prove similar relations for the whole of $K$.
We begin with the multiplicative structure.

\begin{proposition}\label{comp:prop1}
  For all $\alpha, \beta \in K$ and $m \in \Z$, $m \neq 0$, the following holds:
 \begin{enumerate}
   \item
     $\bs(m\alpha) = \bs(\alpha) + d \log(\lvert m \rvert)$.
    \item
      $\bs(\alpha \beta) \leq \bs(\alpha) + \bs(\beta) + C$,
    \item 
      $\bs(\alpha^{-1}) \leq d\bs(\alpha) + C$.
  \end{enumerate}
\end{proposition}

\begin{proof} We write $\alpha = \tilde\alpha/k$ and $\beta = \tilde \beta/l$ with $k$ and $l$ the denominator of $\alpha$ and $\beta$ respectively.
  Note that by the choice of $C$ items (2) and (3) hold for integral elements.
  (1): From the definition of the size it follows that $\bs(m\tilde\alpha) = \bs(\tilde\alpha) + d\log(\lvert m \rvert)$.
  Since the denominator of $m\alpha$ is bounded by $k$ we have
  \[ \bs(m\alpha) \leq \bs(m\tilde\alpha) + d \log(k) = \bs(\alpha) + d \log(\lvert m \rvert). \]
  (2): Since the denominator of $\alpha\beta$ is bounded by $kl$ we obtain
  \[ \bs(\alpha\beta) \leq \bs(\tilde\alpha\tilde\beta) + d\log(kl) \leq \bs(\alpha) + \bs(\beta) + C. \]
  (3): The inverse of $\alpha$ is equal to $k \tilde\alpha^{-1}$. Therefore using (1) we get
  \[ \bs(\alpha^{-1}) = \bs(\tilde\alpha^{-1}) + d \log(k) = d \log(k) + \bs(\tilde\alpha) + C \leq d \bs(\alpha) + C. \qedhere \]
\end{proof}

We now investigate the additive structure.

\begin{proposition}\label{size:elementsadd}
  If $\alpha$ and $\beta$ are elements of $K$ then $\bs(\alpha + \beta) \leq 2( \bs(\alpha) + \bs(\beta))$.
\end{proposition}

\begin{proof}
  It is easy to see that $\bs(\alpha +\beta) \leq \bs(\alpha) + \bs(\beta)$ if $\alpha$ and $\beta$ are integral elements.
  Now write $\alpha = \tilde\alpha/k$ and $\beta = \tilde\beta/l$ with $k$ and $l$ the denominator of $\alpha$ and $\beta$ respectively.
  Then we obtain  $\bs(l\tilde\alpha + k\tilde\beta) \leq \bs(\tilde\alpha) + \bs(\tilde\beta) + d\log(k) + d \log(l) = \bs(\alpha) + \bs(\beta)$ and finally
  \[ \bs(\alpha+\beta) \leq \bs(l\tilde\alpha + k\tilde\beta) + d \log(kl) \leq 2 (\bs(\alpha) + \bs(\beta)). \qedhere \]
\end{proof}

Finally we need the mixed operation between ideals and elements.

\begin{proposition}\label{size:elementideal}
  Let $\alpha \in K$ and $\mathfrak a \subseteq K$ be a fractional ideal. Then $\bs(\alpha \mathfrak a) \leq \bs(\mathfrak a) + d^2 \bs(\alpha) + dC$.
\end{proposition}

\begin{proof}
  We consider first the integral case $\alpha \in \mathcal O_K$ and $\mathfrak a \subseteq \mathcal O_K$.
  Using Inequalities~(\ref{ineq:1}) and (\ref{ineq:3}) the minimum of the principal ideal $(\alpha)$ can be bounded by $C_1^d \|\alpha\|_\infty^d$.
  Thus we have
  \[ \bs(\alpha\mathfrak a) = d^2 \log(\min(\alpha\mathfrak a)) \leq d^2 \log(\min(\alpha)) + d^2 \log(\min(\mathfrak a)) \leq d^2 \bs(\alpha) + \bs(\mathfrak a) + dC. \]
  Now let $\alpha = \tilde\alpha/k$ and $\mathfrak a = \tilde{\mathfrak a}/l$ with $k$ and $l$ the denominator of $\alpha$ and $\mathfrak a$ respectively.
  Using the integral case we obtain
  \begin{align*} \bs(\alpha\mathfrak a) \leq \bs(\tilde\alpha\tilde{\mathfrak a}) + d^2 \log(kl) &\leq \bs(\tilde{\mathfrak a}) + d^2 \log(l) + d^2 \bs(\tilde\alpha) + d^2 \log(k) + dC \\ 
    &= \bs(\mathfrak a) + d^2 \bs(\alpha) + dC. \qedhere\end{align*}
\end{proof}

\subsection*{Calculating in $K$.}

In this section, we evaluate the complexity of the basic operations performed during the pseudo-HNF algorithm.
To simplify the representation of complexity results, we use soft-Oh notation $\OT$: We have $f \in \OT(g)$ if and only if there exists $k \in \Z_{>0}$ such that $f \in O(g (\log(g))^k)$.
We multiply two integers of bit size $B$ with complexity in $\OT(B)$ using the Sch\"onhage--Strassen algorithm \cite{Schoenhage1971}.
While the addition of such integers is in $O(B)$, their division has complexity in $\OT(B)$.
\par
As most of our algorithms are going to be based on linear algebra over rings,
mainly $\Z$, we start be collecting the complexity of the used algorithms.
The basic problem of determining the unique solution $x\in \Q^n $ to the equation $Ax=b$ with $A \in \Z^{n \times n}$ non-singular, $b \in \Z^{n}$ can be done using Dixon's $p$-adic algorithm \cite{Dixon1982} in $\OT(n^3(\log (\lvert A \rvert) + \log(\lvert b \rvert)))$.
\par
As we represent integral ideals using the HNF basis, the computation of this form is at the heart of ideal arithmetic. 
Note, that in contrast to the standard case in the literature \cite{Hafner1991,Storjohann1996} we do not want to state the complexity in terms of
the determinant (or multiples thereof) but in terms of the elementary divisors.
As we will see, in our applications, we always know small multiples of the
elementary divisors and thus obtain tighter bounds. Important to the 
algorithms is the notion of a Howell form of a matrix as defined in \cite{Howell1986}. The Howell form generalizes the Hermite normal form to $\Z/\lambda \Z$ and
restores uniqueness in the presence of zero divisors.
For a matrix $A\in \Z^{n \times m}$ of rank $m$ we denote by $\HNF(A)$ the 
unique Hermite form of the matrix (with the off-diagonal elements reduced
into the positive residue system modulo the diagonal), while $\HOW_\lambda(A)$ will
denote the Howell form for $A\in (\Z/\lambda\Z)^{n \times m}$. In \cite{Storjohann1998} a naive algorithm is given that computes $\HOW_\lambda(A)$
in time $O(m^2\max(n,m))$ operations in $\Z/\lambda \Z$.
We also need the following facts:
\begin{lemma}\label{lem:claus}
Let $A\in \Z^{n\times m}$ and $\lambda\in \Z$ such that $\lambda \Z^m
\subseteq [A]_\Z$ where $[A]_\Z$ denotes the $\Z$-module generated by the
rows of $A$. Then the following holds:
\begin{enumerate}
  \item 
  We have
  \[ \HNF(A) = \HNF \left( \begin{array}{c} A \\ \hline \lambda I_m \end{array}\right).\]
\item We have $\HNF(A) = \HOW_{\lambda^2}(A)$,
that is, the canonical lifting of the Howell form over $\Z/\lambda^2\Z$
yields the Hermite form over $\Z$.
\end{enumerate}
\end{lemma}
\begin{proof}
Since, by assumption, 
\[ [A]_\Z = \left[ \begin{array}{c} A \\ \hline \lambda I_m \end{array} \right]_\Z \] and the
Hermite form is an invariant of the module, the first claim is clear.

To show the second claim, it is sufficient to show that the reduction of
$\HNF(A)$ modulo $\lambda^2$ has all the properties of the Howell form.
Once this is clear, the claim follows from the uniqueness of the Howell form
as an invariant of the $\Z/\lambda^2\Z$ module and the fact that all entries
in $\HNF(A)$ are non-negative and bounded by $\lambda$.
The only property of the Howell form that needs verification, is the last
claim: any vector in $[A]_{\Z/\lambda^2\Z}$ having first coefficients
zero is in the span of the last rows of the Howell form. This follows directly
from the Hermite form: any lift of such a vector is a sum of a vector
in $\lambda^2\Z$ and an element in $[A]$ starting with the same number of
zeroes as the initial element. Such an element is clearly in the span of the
last rows of $\HNF(A)$ since the Hermite form describes a basis and the linear
combination carries over modulo $\lambda^2$.
The other properties of the Howell form are immediate: the reduction modulo
the diagonal as well as the overall shape is directly inherited from the Hermite form.
The final property, the normalization of the diagonal namely 
to divide $\lambda^2$ follows too from the Hermite form: since $\lambda\Z^m$
is contained in the module, the diagonal entries of the Hermite form
have to be divisors of $\lambda$, hence of $\lambda^2$. We note, that the reason
we chose $\lambda^2$ over $\lambda$ is to avoid problems with vanishing
diagonal elements: as all diagonal entries of the Hermite form are
divisors of $\lambda$, none of them can vanish in $\Z/\lambda^2\Z$.
\end{proof}

We can now derive the complexity of the HNF computation in terms of $\lambda$.

\begin{corollary}\label{cost:zhnf}
  Let $A \in \Z^{n\times m}$ be a matrix and $\lambda \in \Z$ such that $\lambda \Z^m \subseteq [A]_\Z$.
  Then the Hermite normal form of $A$ can be computed with complexity in $\OT(mn\log(\lvert A \rvert) + m^2\max(m,n)\log(\lvert \lambda \rvert))$.
\end{corollary}

\begin{proof}
  The Lemma \ref{lem:claus} links the Hermite normal form to the 
  Howell form, while Storjohann's naive algorithm \cite{Storjohann1998}
  will compute the Howell form with the complexity as stated.
\end{proof}

We will see, that in our applications, we naturally know and control a multiple
of the largest elementary divisor, hence we can use this rather than
the determinant in our complexity analysis.
\par
Note that due to Storjohann and Mulders \cite{Storjohann1998} there exists asymptotically fast algorithms for computing the Howell form based on fast matrix multiplication.
Since our pseudo-HNF algorithm is a generalization of a non-asymptotically fast HNF algorithm over 
the integers and eventually we want to compare our pseudo-HNF algorithm with the absolute HNF algorithm it is only reasonable to not use asymptotically fast algorithms for the underlying element and ideal arithmetic.
\par
Concerning our number field $K$, we take the following precomputed data for granted:

\begin{itemize}
  \item
    An integral basis $\Omega = (\omega_i)_i$ of the maximal order $\OK$ satisfying $\omega_1 = 1$.
  \item
    The structure constants $M = (m_{i,j}^k)_{i,j,k}$ of $\Omega$.
  \item
    The matrix $DT^{-1}$, where $T = (\operatorname{Tr}(\omega_i\omega_j))_{i,j}$ and $D$ is the denominator of $T^{-1}$. Moreover using \cite[Theorem 3]{Fieker2010} we compute a LLL-reduced $2$-element representation $(\delta_1,\delta_2)$ of the ideal $\mathfrak B$ generated by the rows of $DT^{-1}$ with the property
    \[ \|\delta_i\| \leq 4 (2^{\frac d 2})^8 \left|\Delta_K\right|^{\frac 2 d} \left(C_1\right)^4, \quad \text{i.e., } \bs(\delta_i) \in \OT\left(\log( \lvert \Delta_K \rvert) + C\right) \]
    for $i = 1,2$. In addition we compute the regular representations $M_{\delta_1}$ and $M_{\delta_2}$.
  \item
    A primitive element of $K$ with minimal polynomial $f = X^d + \sum_{i=0}^{d-1} a_i X^i \in \Z[X]$, such that $\log(\max_i \lvert a_i \rvert) \leq C$ and $\log(\lvert \operatorname{disc}(f) \rvert) \in \OT(C)$.
    Such an element can be found as follows:
    By a theorem of Sonn and Zassenhaus \cite{Sonn1967} there exist $\varepsilon_1,\dotsc,\varepsilon_d \in \{0,1\}$ such that $\alpha = \sum_{i=1}^d \varepsilon_i \omega_i \in \mathcal O_K$ is a primitive element of the field extension $\mathbb Q \subseteq K$.
    Note that with the currently known methods finding such an element is exponentially costly with respect to $d$.
    Applying the $d$ embeddings $\sigma_j$ we obtain
    \[ \lvert \sigma_j(\alpha)\rvert \leq d \max_{i} \lvert \sigma_j(\omega_i) \rvert \leq d \max_i \left\| \omega_i \right\| \leq d C_1. \]
    Using these estimates for the conjugates of $\alpha$ we get the following bound on the coefficients of the minimal polynomial $f = X^d + \sum_{i=0}^{d-1} a_i X^i \in \mathbb Z[X]$ of $\alpha$: Since the elements $\sigma_j(\alpha)$, $1 \leq j \leq d$, are exactly the roots of $f$ we obtain
    \[ \lvert a_i \rvert = \lvert s_i ( \sigma_1(\alpha), \dotsc,\sigma_d(\alpha))\vert \leq {d \choose i} \max_{j} \lvert \sigma_j(\alpha)\rvert^i \leq d^d \max_j \lvert \sigma_j(\alpha)\rvert^d \leq d^d d^d C_1^d,
    \]
    for $0 \leq i \leq d-1$, where $s_i$ denotes the elementary symmetric polynomial of degree $i$.
    Therefore the height $\lvert f \rvert = \max_{i} \lvert a_i \rvert$ of $f$ can by estimated by 
    \[ \log (\lvert f \rvert) = \max_i \log(\lvert a_i \rvert) \leq 2d \log(d) + d \log(C_1) \leq C, \]
    As we have a bound for the absolute values of its roots, we can moreover derive the following estimate for the discriminant of $f$:
    \[ \lvert \operatorname{disc}(f) \vert = \prod_{i<j} \lvert \sigma_i(\alpha) - \sigma_j(\alpha)\rvert^2 \leq \lvert \max_j 2 \sigma_j(\alpha) \rvert ^{d^2} \leq 2^{d^2} \max_j \lvert \sigma_j(\alpha)\rvert^{d^2}.  \]
    Taking logarithms on both sides we obtain 
    \[ \log (\lvert \operatorname{disc}(f) \vert) \in O(\log( d^2 \max_j \lvert\sigma_j(\alpha)\rvert) \subseteq O(d^2(\log(d)+\log(C_1)) \subseteq \tilde O(C).\]
\end{itemize}

We do not impose any further restrictions on our integral basis $\Omega$.
All dependency on $\Omega$ is captured by $C = C_{\Omega}$.

\subsubsection*{Field arithmetic}

During our pseudo-HNF computation we need to perform additions, multiplications, and inversions of elements of $K$.
Although algorithms for these operations are well known (see \cite{Cohen1993,Belabas2004}) and many implementations can be found, there is a lack of references on the complexity.
While multiplication in $\mathcal O_K$ was investigated by Belabas \cite{Belabas2004}, all the other operations are missing.
We address the complexity issues in the rest of this section and begin with the additive structure.

\begin{proposition}
  Let $\alpha,\beta \in K$ and $m \in \Z$. We can
  \begin{enumerate}
    \item
      compute the product $m\alpha$ with complexity in $\OT(\bs(\alpha) + d \log(\lvert m \rvert))$.
    \item
      compute the quotient $\alpha/m$ with complexity in $\OT(\bs(\alpha) + \log(\lvert m \rvert))$.
    \item
      compute the sum $\alpha + \beta$ with complexity in $O(\bs(\alpha) + \bs(\beta))$.
  \end{enumerate}
\end{proposition}

\begin{proof}
  Let us write $\alpha = \tilde\alpha/k$ and $\beta = \tilde \beta /l$ with $k$ and $l$ the denominator of $\alpha$ and $\beta$ respectively.
  \par
  (1): Computing the GCD $g$ of $m$ and $k$ as well as $k/g$ and $m/g$ have complexity in $\OT(\bs(\alpha)/d + \log(m))$.
  This is followed by computing $(m/g)\tilde\alpha$ which has complexity in $\OT(\bs(\alpha) + d \log(m))$ and dominates the computation.
  \par
  (2): Let $(a_1,\dotsc,a_d)$ be the coefficient vector of $\tilde\alpha$ and $g = \operatorname{GCD}(m,a_1,\dotsc,a_d)$.
  The quotient $\alpha/m$ is then given by $(\tilde\alpha/g)/(k \cdot m/g)$. 
  As the costs of computing $g$ are in $\OT(\log(\lvert m \rvert) + d \log(\|\tilde\alpha\|_\infty))$ and the products can be computed in $\OT(d \log(\|\tilde\alpha\|_\infty))$ and $\OT(\log(\lvert m \rvert) + \log(k))$ the claim follows.
  \par
  (3): The complexity obviously holds for integral elements.
  By (1) the computation of $l\tilde\alpha$ and $k\tilde\alpha$ has complexity in $\OT(\bs(\alpha) + \bs(\beta))$ and the complexity of adding $l\tilde\alpha$ and $k\tilde\beta$ is in $\OT(\bs(\alpha)+\bs(\beta))$.
  Computing $kl$ has complexity in $\OT(\bs(\alpha)/d + \bs(\beta)/d)$.
  The last thing we have to do is making sure that the coefficients of the numerator and the denominator are coprime.
  This is done by $d$ GCD computations and $d$ divisions with complexity in $\OT(d(\bs(\alpha)/d + \bs(\beta)/d))$.
\end{proof}

\begin{proposition}\label{cost:elements}
  Let $\alpha,\beta,\alpha_1,\dotsc,\alpha_n \in K$, $\gamma \in \mathcal O_K$ an integral element and $m \in \Z$. We can
  \begin{enumerate}
    \item
     compute the regular representation $M_\gamma$ of $\gamma$ with complexity in $\OT(d^2 \bs(\gamma) + d^2 C)$.
    \item
      compute the product $\alpha \beta$ with complexity in $\OT(d \bs(\alpha)+d\bs(\beta) + d C)$ if the regular representation of the numerator of $\alpha$ is known.
    \item
      compute the product $\alpha \beta$ with complexity in $\OT(d^2 \bs(\alpha) + d\bs(\beta) + d^2 C)$
    \item
      compute the products $\alpha\alpha_i$, $1 \leq i \leq n$, with complexity in $\OT(d(d+n)\bs(\alpha) + dn\max_i\bs(\alpha_i) + d(d+n) C)$.
    \item
      compute the inverse $\alpha^{-1}$ with complexity in $\OT(d^2 \bs(\alpha) + d^2 C)$ if $\alpha \neq 0$.
  \end{enumerate}
\end{proposition}

\begin{proof}
  Let us write $\alpha = \tilde\alpha/k$ and $\beta = \tilde \beta /l$ with $k$ and $l$ the denominator of $\alpha$ and $\beta$ respectively.
  \par
  (1): If $(c_1,\dotsc,c_d) \in \Z^{d}$ denotes the coefficient vector of $\gamma$, the regular representation is given by
  \[ M_\gamma = \Biggl( \sum_{j=1}^d c_j m_{ij}^k\Biggr)_{i,k}. \]
  Thus computing $M_\gamma$ involves $d^3$ multiplications (and additions) and the overall complexity is in $\tilde O(d^2 \bs(\gamma) + d^3 \log(C_3)) = \OT(d^2 \bs(\gamma) + d^2 C)$.
  \par
  (2): Let $(a_1,\dotsc,a_d)$ and $(b_1,\dotsc,b_d)$ be the coefficient vectors of $\tilde\alpha$ and $\tilde\beta$ respectively. 
  The coefficients of the product $\tilde \alpha \tilde \beta = \sum_{i=1}^d c_i \omega_i $ are given by
  \[ (c_1,\dotsc,c_d) = (b_1,\dotsc,b_d) M_{\tilde\alpha}. \]
  Hence the product is obtained by $d^2$ multiplications.
  As the matrix $M_{\tilde\alpha}$ satisfies $\log(\lvert M_{\tilde\alpha} \rvert)\in 
\tilde O\left(\bs(\tilde\alpha)/d + {C}/{d}\right)$ this has complexity in $\OT (d \bs(\tilde\alpha) + d \bs(\tilde\beta) + dC)$.
  Since taking care of denominators is less expensive this step dominates the computation.
  \par
  (3):
  Use (1) and (2).
  \par
  (4):
  We fist evaluate the complexity of inverting the integral element $\tilde \alpha$.
  In this case the coefficients $b_1,\dotsc,b_n$ of the element $\delta \in K$ with $\tilde \alpha \delta = 1$ satisfy
  \[ (b_1,\dotsc,b_d) M_{\tilde\alpha} = (1,0,\dotsc,0).\]
  Thus inverting $\tilde\alpha$ boils down to calculating the regular representation of $\tilde\alpha$ and finding the unique rational solution of a linear system of $d$ integer equations.
  By (1) the computation of $M_{\tilde\alpha}$ has complexity in $\OT(d^2 \bs(\tilde\alpha) + d^2 C)$ and the entries of $M_{\tilde\alpha}$ satisfy $\log (\lvert M_{\tilde\alpha} \rvert) \in O\left({\bs(\tilde\alpha)}/{d} + {C}/{d}\right)$.
  Using Dixon's algorithm solving the system then has complexity in $\OT(d^2 \bs(\alpha) + d^2 C)$.
  Now the inverse of $\alpha$ is given by $\alpha^{-1} = k \tilde\alpha^{-1}$.
  Since $\bs(\tilde\alpha^{-1}) \leq d\bs(\alpha)+ C$ the complexity to compute $k\tilde\alpha^{-1}$ is in $\OT(d \bs(\alpha)+ C)$.
\end{proof}

\subsection*{Ideal arithmetic}

By definition integral ideals are represented by their unique HNF with respect to the fixed integral basis.
Therefore operations with ideals are mainly HNF computations which are accelerated by the availability of a multiple of the corresponding largest elementary divisor.
More precisely let $\mathfrak a$ be an integral ideal with HNF $A \in \Z^{d \times d}$.
As $\min(\mathfrak a)$ is an element of $\mathfrak a$ we know that $\min(\mathfrak a)\omega_i \in \mathfrak a$ for all $1 \leq i \leq d$.
On the side of the $\Z$-module structure this implies $\min(\mathfrak a)\Z^d \subseteq [A]_\Z$ allowing us to work modulo $\min(\mathfrak a)^2$ during the HNF computation by Lemma~\ref{lem:claus}.
The following lemma for computing the sum of ideals illustrates these ideas.

\begin{lemma}\label{cost:idealsum}
  Let $\mathfrak a$ and $\mathfrak b$ be fractional ideals and $m \in \Z$. We can
  \begin{enumerate}
    \item
      compute $m\mathfrak a$ with complexity in $\OT(\bs(\mathfrak a) + d^2 \log(\lvert m \rvert))$.
    \item
      compute the sum $\mathfrak a + \mathfrak b$ with complexity in $\OT(d (\bs(\mathfrak a)+\bs(\mathfrak b)))$.
  \end{enumerate}
\end{lemma}

\begin{proof}
  We write $\mathfrak a = \tilde{\mathfrak a}/k$ and $\mathfrak b = \tilde{\mathfrak b}/l$ with $k$ and $l$ the denominator of $\mathfrak a$ and $\mathfrak b$ respectively.
  \par
  (1): We first have to compute the GCD $g$ of $m$ and $k$.
  Together with the division of $k$ and $m$ by $g$ this has complexity in $\OT(\log(\lvert m \rvert) + \log(k))$.
  Finally we have to multiply the HNF matrix of $\tilde{\mathfrak a}$ with $m/g$ taking $d^2$ multiplications with integers of size bounded by $\bs(\tilde{\mathfrak a})/d^2 + \log(\lvert m \rvert)$.
  In total we obtain a complexity in $\OT(\bs(\mathfrak a) + d^2 \log(\lvert m \rvert))$.
  \par
  (2):
  We first consider the case of integral ideals $\tilde{\mathfrak a}$ and $\tilde{\mathfrak b}$.
  The HNF basis of $\tilde{\mathfrak a}+\tilde{\mathfrak b}$ is obtained by computing the HNF of the concatenation $(M_{\tilde{\mathfrak a}}^t | M_{\tilde{\mathfrak b}}^t)^t$.
  As the minimum of $\tilde{\mathfrak a} +\tilde{\mathfrak b}$ divides $\operatorname{GCD}(\min(\tilde{\mathfrak a}),\min(\tilde{\mathfrak b}))$ by Corollary~\ref{cost:zhnf} this computation can be done with complexity in
  \begin{align*}   & \OT((2d ) d(\log(\min(\tilde{\mathfrak a})) + \log(\min(\tilde{\mathfrak b}))) + (2d)d^2 \log(\operatorname{GCD}(\min(\tilde{\mathfrak a}),\min(\tilde{\mathfrak b})))) \\ 
    \subseteq  & \OT(d\min(\bs(\tilde{\mathfrak a}),\bs(\tilde{\mathfrak b}))).\end{align*}
  Now consider the fractional case.
  By (1) and the integral case computing $l\tilde{\mathfrak a} + k\tilde{\mathfrak b}$ has complexity in $\OT(d(\bs(\mathfrak a) + \bs(\mathfrak b)))$.
  Since this dominates the denominator computation we obtain an overall complexity as claimed.
\end{proof}

\begin{proposition}\label{cost:idealrest}
  Let $\alpha \in K$ and $\mathfrak a,\mathfrak b \subseteq \mathcal O_K$ be integral ideals. We can
  \begin{enumerate}
    \item
      compute $\mathfrak a \mathfrak b$ with complexity in $\OT(d^2 \bs(\mathfrak a) + d^2 \bs(\mathfrak b) + d^3 C)$.
    \item
      compute $\alpha \mathfrak a$ with complexity in $\OT(d^3 \bs(\alpha) + d \bs(\mathfrak a) + d^2 C)$.
 \end{enumerate}
\end{proposition}

\begin{proof}
  We write $\mathfrak a = \tilde{\mathfrak a}/k$, $\mathfrak b = \tilde{\mathfrak b}/l$ and $\alpha = \tilde\alpha/m$ with $k$, $l$ and $m$ the denominator of $\mathfrak a$, $\mathfrak b$ and $\alpha$ respectively.
  \par
  (1): As $\mathfrak a \mathfrak b = \tilde{\mathfrak a}\tilde{\mathfrak b}/(kl)$ we first evaluate the complexity of computing $\tilde{\mathfrak a}\tilde{\mathfrak b}$.
  Denoting by $(\alpha_i)_i$ and $(\beta_j)_j$ the HNF bases of $\tilde{\mathfrak a}$ and $\tilde{\mathfrak b}$ respectively we know that $\bs(\alpha_i) \leq \bs(\tilde{\mathfrak a})/d$ and $\bs(\beta_i) \leq \bs(\tilde{\mathfrak b})/d$ respectively.
  The $d^2$ elements $(\alpha_i \beta_j)_{i,j}$ form a $\Z$-generating system of $\tilde{\mathfrak a}\tilde{\mathfrak b}$ and their computation has complexity in
  \[ \OT(d^3(\bs(\tilde{\mathfrak a})/d + \bs(\tilde{\mathfrak b}/d) + d^3 C)) = \OT(d^2\bs(\tilde{\mathfrak a}) + d^2 \bs(\tilde{\mathfrak b}) + d^3 C). \]
  The matrix $M$ of this generating system then satisfies $\log(\lvert M \rvert)\leq \bs(\tilde{\mathfrak a})/ d^2 + \bs(\tilde{\mathfrak b})/d^2 + C/d$.
  As the minimum of $\tilde{\mathfrak a}\tilde{\mathfrak b}$ divides $\min(\tilde{\mathfrak a})\min(\tilde{\mathfrak b})$ the final HNF computation has complexity in 
  \[ \OT(d^2 \bs(\tilde{\mathfrak a}) + d^2 \bs(\tilde{\mathfrak b}) + d^2 C). \]
  As denominator computation is dominated by these steps the claim holds.
  \par
  (2): If we denote the HNF basis of $\tilde{\mathfrak a}$ by $(\alpha_i)_i$ we know that $(\tilde\alpha \alpha_i)_i$ forms a $\Z$-generating system of the ideal $\tilde\alpha \tilde{\mathfrak a}$.
  Computing the $d$ products $\tilde\alpha \alpha_i$ for $1 \leq i \leq d$ has complexity in $\OT(d^2 \bs(\tilde\alpha) + d \bs(\tilde{\mathfrak a}) + d^2 C)$ since we have to compute the regular representation of $\tilde\alpha$ only once.
  If $M$ denotes the matrix corresponding to this generating system of $\tilde\alpha\tilde{\mathfrak a}$ we know that $\log( \lvert M \rvert) \leq \bs(\tilde\alpha)/d + \bs(\tilde{\mathfrak a})/d^2 + C/d$.
  Before computing the HNF matrix, we take care of the denominator.
  Computing $kl$, the GCD of $kl$ and the entries of the matrix $M$ and dividing $kl$ and $M$ by the GCD has complexity in $\OT(d\bs(\alpha) + \bs(\mathfrak a) + dC)$.
  As we know the regular representation of $\tilde\alpha$ we also know the minimum of the principal ideal $(\alpha)$.
  In particular we know $\min((\tilde\alpha))\min(\tilde{\mathfrak a})$ which is a multiple of $\min(\tilde{\alpha}\tilde{\mathfrak a})$.
  Using the estimate $\bs((\tilde\alpha)) \leq d^2 \bs(\tilde\alpha) + dC$ (see proof of Proposition~\ref{size:elementideal}) and Corollary~\ref{cost:zhnf} the final HNF can be computed with complexity in 
  \[ \OT(d \bs(\mathfrak a) + d^3 \bs(\alpha) + d^2C). \qedhere\]
\end{proof}

Finally we need to invert ideals.
We use a slightly modified version of \cite[Algorithm 5.3]{Belabas2004} (which itself is a modified version of \cite[Algorithm 4.8.21]{Cohen1993}), exploiting the fact that
\[ \mathfrak a^{-1} = \left\{ \alpha \in K \,\middle| \,\operatorname{Tr}(\alpha \mathfrak D^{-1}\mathfrak a) \subseteq \Z \right\}, \]
where $\mathfrak D$ denotes the different of $K$. 
Recall that $\mathfrak D^{-1}$ is a fractional ideal with (fractional) basis matrix $T^{-1} \in \Q^{d\times d}$, where $T = (\operatorname{Tr}(\omega_i\omega_j))_{i,j}$.
In order to evaluate the complexity of ideal inversion we need a bound on the size of $m T^{-1}$ where $m$ denotes the denominator of $T^{-1}$, that is, $m = \min(\mathfrak D)$.
Since by Cramer's rule we know that $\lvert m T^{-1} \rvert \leq d^d \lvert T \rvert^d$ it remains to consider $\lvert T \rvert$.
By definition the trace of an element $\alpha \in K$ is given by the trace of its regular representation, $\operatorname{Tr}(\alpha) = \operatorname{Tr}(M_{\alpha})$.
In case of a basis element $\alpha = \omega_k$ for some $1 \leq k \leq d$ the entries of $M_{\alpha}$ are just structure constants $m_{ij}^k$ and therefore $\lvert \operatorname {Tr} (\omega_k)\rvert \leq d C_3$.
Applying this to $\operatorname{Tr}(\omega_i\omega_j)$ for $1 \leq i,j\leq d$ yields
\[ \lvert \operatorname{Tr}(\omega_i\omega_j) \rvert \leq \sum_{k=1}^d \lvert m_{ij}^k \rvert \lvert \operatorname{Tr}(\omega_k) \rvert \leq d^2 C_3^2 \]
and therefore
\[ \log(\lvert mT^{-1}\rvert) \leq 2d \log(d) + 2d \log(C_3) \in O(C) .\]
In addition note that $\min(\mathfrak D)$ divides the norm of $\mathfrak D$, which is just $\lvert \Delta_K \rvert$.

\begin{proposition}\label{cost:idealinv}
  Let $\mathfrak a$ be a fractional ideal.
  Then we can compute $\mathfrak a^{-1}$ with complexity in $\OT(d \bs(\mathfrak a) + d^3 \log( \left|\Delta_K \right|) + d^2 C)$. 
\end{proposition}

\begin{proof}
  We use the same notation as in the preceding discussion.
  Let us first consider the integral case $\mathfrak a \subseteq \mathcal O_K$.
  Recall that the denominator of $\mathfrak a^{-1}$ is just $\min(\mathfrak a)$ and need not be computed.
  Denote by $(\alpha_i)_i$ the HNF basis of $\mathfrak a$ and by $\mathfrak B$ the integral ideal $m \mathfrak D^{-1}$.
  We first have to compute $\mathfrak a\mathfrak B$.
  Using the precomputed $2$-element representation $\mathfrak B = (\delta_1,\delta_2)$ this amounts to compute $2d$ products $\alpha_i \delta_j$, $1 \leq i \leq d$, $1 \leq j \leq 2$.
  As we have also precomputed the regular representation of $\delta_1$ and $\delta_2$ this has complexity in $\OT(d\bs(\mathfrak a) + d^2 \log(\lvert \Delta_K \rvert) + d^2 C)$ and yields a matrix $M \in \Z^{2d\times d}$ with $\log(\lvert M \rvert) \leq \bs(\mathfrak a)/d^2 + \log(\lvert \Delta_K)/d + C/d$.
  The cost of computing the HNF $H$ of $M$ is therefore in $\OT(d \bs(\mathfrak a) + d^3 \log(\lvert \Delta_K\rvert)+  dC)$, where we use that the minimum of $\mathfrak a \mathfrak B$ divides the $\min(\mathfrak a) \lvert \Delta_K \rvert$.
  A transposed basis matrix of the numerator of $\mathfrak a^{-1}$ is then obtained as the solution $X \in \Z^{d\times d}$ of the equation $HX = \min(\mathfrak a)(mT^{-1})$.
  Note that the triangular shape of $H$ allows us the recover $X$ by back substitution. 
  Since $\min(\mathfrak a)^2$ is contained in the span of $X$ we can work modulo $\min(\mathfrak a)^2$. 
  The estimates $\log(\lvert H \rvert)\leq \log(\min(\mathfrak a)\lvert \Delta_K \rvert)$ and $\log(mT^{-1}) \in O(C)$ show that the initial reduction has complexity in $\OT(\bs(\mathfrak a) + d^2 \log(\lvert \Delta_K \rvert) + d^2 C)$.
  For each column of $X$ the back substitution itself then has a complexity in $\OT(d^2 \min(\mathfrak a))$ yielding a complexity of $\OT(d \bs(\mathfrak a))$ in total for obtaining $X$.
  Finally we need to compute the HNF of $X^t$ which has complexity in $\OT(d^2 \log(\lvert X \rvert) + d^3 \log(\min(\mathfrak a))) \subseteq \OT(d\bs(\mathfrak a))$.
  \par
  Now let $\mathfrak a = \tilde{\mathfrak a}/k$ be a fractional ideal with denominator $k$.
  As $\bs(\tilde{\mathfrak a}^{-1}) \leq 2 \bs(\tilde{\mathfrak a})$ the computation of $k \tilde{\mathfrak a}^{-1}$ has complexity in $\OT(\bs(\tilde{\mathfrak a}) + d^2 \log(k)) = \OT(\bs(\mathfrak a))$ and the claim follows.
\end{proof}

%%%%%%%%%%%%%%%%%%%%%%%%%%%%%%%%%%%%%%%%%%%%%%%%%%%%%%%%%%%%%%%%%%%%%%%%%%%%%%%%
%
%  Normalization and reduction
%
%%%%%%%%%%%%%%%%%%%%%%%%%%%%%%%%%%%%%%%%%%%%%%%%%%%%%%%%%%%%%%%%%%%%%%%%%%%%%%%%

\section{Normalization and reduction}

There are different strategies for dealing with coefficient explosion during classical Hermite normal form algorithms over $\Z$.
One strategy, which is used by Hafner and McCurley \cite{Hafner1991} exploits the fact that the whole computation can be done modulo some multiple of the determinant of the associated lattice (in case of a square non-singular matrix this is just the determinant of the matrix).
Fortunately the same holds for the pseudo-HNF over Dedekind domains and therefore we are allowed to use reduction modulo (different) integral ideals involving the determinantal ideal.
Unfortunately these ideals are in general not generated by a single rational integer, making the notion of reduction more difficult.
We will use the approach of Cohen \cite[Algorithm 2.12]{Cohen1996} with a different reduction algorithm and provide a rigorous complexity analysis.
The reduction is accompanied by a normalization algorithm, which bounds the size of the coefficient ideals and heavily depends on lattice reduction.
The output of both algorithms, reduction and normalization, depends on the size of the lattice reduced basis and the smaller the lattice basis the smaller the output.
There are various lattice reduction algorithms and in general the smaller the basis the worse the complexity of the algorithm.
Thus one has to balance between smallness and efficiency.
Instead of the L$^2$ algorithm of Nguyen and Stehl{\'e} \cite{Nguyen2009}, which has complexity quadratic in the size of the input, we rely on the nearly linear \~L$^1$-algorithm of Novocin, Stehl\'{e} and Villard which provides a lattice basis satisfying a weakened LLL condition.
More precisely, for $\Xi = (\delta,\eta,\delta)$ with $\eta \in [\frac 1 2,1)$, $\theta \geq 0$ and $\delta \in (\eta^2,1]$, the notion of an $\Xi$-LLL reduced basis is defined in \cite{Chang2012}.
Setting $\ell = ({\theta\eta + \sqrt{(1 + \theta^2)\delta - \eta^2}}){(\delta - \eta^2)^{-1}}$ it is proved in \cite[Theorem 5.4]{Chang2012} that an $\Xi$-LLL reduced basis $(b_1,\dotsc,b_n)$ of a lattice $L$ of rank $n$ in an Euclidean space satisfies
\begin{equation}\label{eq:LLL}
  \begin{aligned}
  \left\| b_1 \right\| &\leq \ell^{n-1} \lambda(L), \\
  \left\| b_1 \right\| &\leq \ell^{(n-1)/2} \lvert \det(L) \rvert^{1/n}, \\
  \prod_{1 \leq j \leq n} \left\| b_j \right\| &\leq \ell^{n(n-1)/2} \lvert\det(L)\rvert,
\end{aligned}
\end{equation}
where $\det(L)$ resp. $\lambda(L)$ denotes the determinant resp. the first minimum of the lattice $L$.
Using this weakened LLL condition Novocin, Stehl\'{e} and Villard (\cite{Novocin2011}) construct an algorithm, \~L$^1$, with the following property (\cite[Theorem 7]{Novocin2011}):
Given a matrix $B \in \Z^{d\times d}$ with rows $b_1,\dotsc,b_j$ satisfying $\max_j \left\| b_j \right\| \leq 2^\beta$, the $\tilde L^1$ algorithm returns a $\Xi$-LLL reduced basis of the lattice associated to $B$ within $\OT(d^5 \beta)$ operations.

\subsection*{Rounded lattice reduction}

Since the $\tilde L^1$ algorithm operates only on integral input, we now describe how to use this algorithm in our case, where the input is a lattice with real basis.
We closely follow the ideas and arguments of \cite[Section 4]{Belabas2004}, where a similar analysis was done for LLL reduction.
Let $G = (T_2(\omega_i,\omega_j))_{1 \leq i, j \leq d} \in \R^{d \times d}$ be the Gram matrix of $T_2$ with respect to the integral basis of $K$.
Let $R R^t$ be the Cholesky decomposition of $G$ and $e \in \Z_{\geq 1}$ an integer such that the integral matrix $R^{(e)} = \lceil 2^e R \rfloor \in \Z^{d\times d}$ has full rank.
Thus for an element $\alpha = \sum_{i=1}^d a_i \omega_i$ of $K$ with coefficient vector $X = (a_1,\dotsc,a_d)$ we have $\left\|\alpha \right\| = \left\| XR\right\|_2$.
We now set $T_2^{(e)}(\alpha) = \left\| XR^{(e)} \right\|$.
Then $T_2^{(e)}$ is an integral approximation of $2^{2e}T_2$ with integral Gram matrix.
While in general the application of lattice reduction with respect to this
approximated form $T_2^{(e)}$ does not yield a reduced basis with
respect to $T_2$, the basis one obtains satisfies size estimates similar to
(4).

\begin{proposition}
Let $L$ be a sublattice of $\mathcal O_K$ and let $(\alpha_i)_{1\leq i \leq d}$ be a $\Xi$-LLL reduced basis for $L$ with respect to $T_2^{(e)}$.
Then
\begin{align*}
\left\| \alpha_1 \right\| &\leq C_{\Omega, d, e} \cdot \ell^{(d-1)/2} \cdot \lvert \det(L, T_2)\rvert^{1/d} \text{  and} \\
 \prod_{1 \leq i \leq n } \left\| \alpha_i \right\| &\leq C_{\Omega, d, e}^d \cdot \ell^{d(d-1)/2} \cdot \lvert \det(L, T_2)\rvert,
 \end{align*}
with a constant $C_{\Omega, d, e}$ depending on the integral basis $\Omega$, the field degree $d$, $e$, and not depending on $L$, and which satisfies $C_{\Omega, d, e} \rightarrow 1$ for $e \rightarrow \infty$. 
Here with $\det(L, T_2)$ we denote the determinant of $L$ with respect to $T_2$.
\end{proposition}

\begin{proof}
The proof is similar to the proof of \cite[Proposition 4.2]{Belabas2004}.
We set $S = (R^{(e)})^{-1}$ and write $R^{(e)} = 2^e R + \varepsilon$ with $\varepsilon \in \R^{d \times d}$.
Let $X_i$ be the coefficient vector of $\alpha_i$ and let $Y_i = X_iR^{(e)}$.
Since the $(\alpha_i)$ are a $\Xi$-LLL reduced basis of $L$ with respect to $T_2^{(e)}$ we have
\[  \prod_{1 \leq i \leq d} \left\| Y_i \right\|_2 = \prod_{1 \leq i \leq d} \sqrt{T_2^{(e)}(\alpha_i)} \leq \ell^{d(d-1)/2} \lvert \det(L, T_2^{(e)}) \rvert.\]
As in \cite{Belabas2004} we have $2^e \left\| \alpha_i \right\| \leq ( 1 + \left\|\varepsilon S \right\|_2)\left\| Y_i \right\|$.
Using the fact that
\[ \det(L, T_2^{(e)}) = \det(L, T_2) \frac{\det(R^{(e)})}{\det(R)} \]
we obtain
\[ \prod_{1 \leq i \leq d} \left\| \alpha_i \right\| \leq (1 + \left\| \varepsilon S \right\|)^d \frac{1}{2^{de}} \prod_{1 \leq i \leq d} \left\| Y_i \right\|_2 \leq ( 1 + \left\| \varepsilon S \right\|_2)^d \left(\frac{\det(R^{(e)})}{2^{ed}\det(R)}\right) \ell^{d(d-1)/2} \lvert\det(L, T_2)\rvert. \]
Now the claim follows from setting
\[ C_{\Omega, d, e} = ( 1 + \left\|\varepsilon S \right\|_2)\left(\frac{\det(R^{(e)})}{2^{ed}\det(R)}\right)^{1/d} \]
and \cite[Corollary 4.3]{Belabas2004}.
\end{proof}

We call the basis $(\alpha_i)$ as in the statement an approximated $\Xi$-LLL reduced basis of $L$.
To use this in our setting, we add the computation of $\lceil 2^e R \rfloor \in \Z^{d \times d}$ to the list of precomputed data. 
We treat $C_{\Omega, d, e}$ as well as $\lvert \lceil 2^e R \rfloor \rvert$ as constants during the complexity analysis.
Moreover to simplify the exposition we replace $\ell$ by $C_{\Omega_,d, e}^{2/(d-1)} \ell$, so that the last two equations of (4) hold for an approximated $\Xi$-LLL reduced basis.
To compute such a basis for an integral ideal $\mathfrak a$ we just have to apply the $\tilde L^1$ algorithm to the matrix $\lceil 2^e R \rfloor M_\mathfrak a$, which has a complexity in $\tilde O(d^5 \log(\min(\mathfrak a)))$.

\subsection*{Reduction with respect to fractional ideals.}
Using the approximated reduced lattices, we now describe how to use this to reduce elements modulo ideals.
We begin with the integral case and assume that $\mathfrak a$ is an integral ideal and $\alpha \in \mathcal O$.
The goal of the reduction algorithm is to replace the element $\alpha$ by $\overline \alpha \in K$ such that $\alpha - \overline \alpha \in \mathfrak a$ and $\overline \alpha$ is small with respect to $\inorm{\mathfrak a}$ and $T_2$-norm.
Let $(\alpha_i)_i$ be a $\Z$-basis of $\mathfrak a$ and $\alpha = \sum_i a_i \alpha_i$ the representation of $\alpha$ in the $\Q$-basis $(\alpha_i)_i$ of $K$.
The element $\overline \alpha$ defined by $\overline \alpha = \sum_i (a_i - \lceil a_i \rfloor) \alpha_i$ satisfies
\[ \alpha - \overline \alpha = \sum_{i=1}^d \lceil a_i \rfloor \alpha_i \in \mathfrak a
\quad \text{and}\quad
\left\| \overline \alpha \right \| \leq \sum_i |a_i - \lceil a_i \rfloor| \left\| \alpha_i \right\| \leq \frac 1 2 \sum_i \left\| \alpha_i \right\| \leq \frac d 2 \max_i \left\| \alpha_i \right\|. \]
Here, as usual, $\lceil a_i \rfloor := \lceil 1+1/2\rceil$ denotes rounding.
By the arithmetic-geometric mean inequality we have
\[ \left\| \alpha_j \right\| \geq \sqrt d \inorm{\alpha_j}^{\frac 1 d} \geq \sqrt d \inorm{\mathfrak a}^{\frac 1 d} \]
for all $1 \leq j \leq d$ and assuming that $(\alpha_i)_i$ is $\Xi$-LLL reduced, we obtain by (\ref{eq:LLL})
\[ \prod_{1 \leq i \leq d} \left \| \alpha_i \right\| \leq \ell^{\frac{d(d-1)} 2} \det(L_\mathfrak a) \]
where $L_\mathfrak a$ denotes the lattice associated to $\mathfrak a$ and $\det(L_\mathfrak a)$ its determinant.
Using both inequalities we obtain
\[ d^{\frac {d-1} 2} \inorm{\mathfrak  a} ^{\frac{d-1} d} \left\| \alpha_j \right\| \leq \prod_{1 \leq i \leq d} \left\| \alpha_i \right\| \leq \ell^{\frac{d(d-1)} 2} \det(L_\mathfrak a) \]
and thus
\begin{align}\label{eq:LLLbasis}
  \left\| \alpha_j \right\| \leq \ell^{\frac{d(d-1)} 2} d^{- \frac{d-1}2} \inorm{\mathfrak a}^{- \frac{d-1} d} \det(L_\mathfrak a)
\leq \sqrt d \ell ^{\frac{d(d-1)} 2} \inorm{\mathfrak a} ^{\frac 1 d} \sqrt{\left| \Delta_K \right|} 
\end{align}
for all $1 \leq j \leq d$.
Hence we are able to bound $\left\| \overline \alpha \right\|$ in terms of $\inorm{\mathfrak a}$.\par
Consider now the fractional case with $\alpha = \beta/k$ and $\mathfrak a = \mathfrak b/l$.
Then the above consideration applied to $l\beta$ and $k\mathfrak b$ yields an element $\overline \alpha$ with
\[ \alpha - \overline \alpha/(kl) \in \mathfrak a \]
and
\[\|\overline\alpha/(kl) \| \leq d^{3/2} \ell ^{d(d-1)/2} \inorm{\mathfrak a} ^{1/d} \sqrt{\left| \Delta_K \right|} \]
To compute $\overline{\alpha}/(kl)$ we proceed as follows.
Denote by $A = (a_1,\dotsc,a_n)$ the coefficient vector of $\beta$ with respect to the integral basis. 
As $l\beta \subseteq \mathfrak b$ there exists $Y \in \Z^d$ such that $LY = A$, that is, $Y$ is the coefficient vector of $l\beta$ with respect to the basis matrix $L \in \Z^{d\times d}$ of $\mathfrak b$.
Dividing by $k$ we obtain $Y/k$, which is then the coefficient vector of $l\beta$ with respect to the basis matrix $kL$ of $k\mathfrak b$.
Finally the coefficient vector of $\overline \alpha /(kl)$ is given by
\[ \frac{1}{kl} (kL)\Bigl(\frac Y k - \Bigl\lceil\frac Y k \Bigr\rceil \Bigr) = \frac 1 l L\Bigl(\frac{Y\bmod k}k\Bigr) = \frac 1 {kl} L(Y\bmod k). \]
This procedure is summarized in Algorithm~\ref{alg:reduction}.

\begin{algorithm}[ht]
  \caption{Reduction modulo integral ideals}
  \begin{algorithmic}[1]\label{alg:reduction}
    \REQUIRE $\alpha\in K$, fractional ideal $\mathfrak a$ of $K$.
    \ENSURE  $\tilde{\alpha}\in K$ such that $\alpha - \tilde{\alpha} \in \mathfrak a$ and
    $\|\tilde{\alpha}\| \leq d^{3/2} \ell ^{{d(d-1)}/ 2} { \inorm{\mathfrak a}} ^{1/d} \sqrt{\left| \Delta_K \right|}$.
    \STATE Let $\alpha = \beta/ k$ and $\mathfrak a = \mathfrak b / l$.
    \STATE Compute an approximated $\Xi$-LLL reduced basis matrix $L \in \Z^{d\times d}$ of $\mathfrak b$ using the $\widetilde{\mathrm{L}}^1$-algorithm.
    \STATE Solve $LY = lA$ for $Y \in \Z^n$, where $A$ is the coefficient vector of $\beta$.
    \STATE Compute $Y\bmod k$ and $Z =  1/(kl) L(Y\bmod k$).
    \RETURN The element corresponding to $Z$.
  \end{algorithmic}
\end{algorithm}

\begin{proposition}\label{prop:reduction}
  Algorithm~\ref{alg:reduction} is correct and has complexity in
\[ \tilde O(d^3 \bs(\mathfrak a) + d^2 \bs(\alpha) + d^3 \log( \left|\Delta_K\right|) + d^3 C). \]
  The size of the output $\tilde \alpha$ satisfies
  \[ \|\tilde{\alpha}\| \leq d^{ \frac 3 2} \ell ^{\frac {d(d-1)}  2} { \inorm{\mathfrak a}} ^{\frac 1 d} \sqrt{\left| \Delta_K \right|}. \]
  Moreover if the approximated reduced basis of the numerator of $\mathfrak a$ is known, then the reduction of $\alpha$ has complexity in 
  \[ \OT(d \bs(\mathfrak a) + d^2 \bs(\alpha) + d^2 C + d^3 \log(\lvert \Delta_K \rvert)).\]
\end{proposition}

\begin{proof}
  As correctness was already shown we just have to do the cost analysis.
  The $\widetilde{\mathrm{L}}^1$-algorithm allows us to compute $L$ with complexity in $\OT(d^5\log(\min(\mathfrak a)))$.
  Write $B_L = \log( \lvert L \rvert)$ and $B_\beta = \log(\|\beta\|_\infty)$.
  Applying Dixon's algorithm to compute $Y$ has costs in $\OT(d^3(B_L + B_\beta + \log(l)))$ and invoking Cramer's rule we see that $\lvert Y \rvert \leq d^d B_L^d B_\beta\log(l)$, that is, $\log(\lvert Y \rvert) \in \OT(dB_L + B_\beta + \log(l))$.
  Therefore the $d$ divisions required to compute $Y \bmod k$ have complexity in $\OT(d(dB_L +  B_\beta + \log(l)))$.
  Since $\lvert Y \bmod k \rvert \leq k$ the matrix vector multiplications consists of $d^2$ multiplications of integers of size bounded by $\OT(B_L + \log(k))$ and the output satisfies ${\log(\lvert L(Y \bmod k) \rvert)} \in \OT(\log(k) + B_L)$.
  Finally the product $kl$, as well as $d$ GCDs and divisions with $L(Y \bmod k)$ need to be computed with complexity in $\OT(d(B_L + \log(k) + \log(l)))$.
  Without the computation of the approximated reduced basis we have in total a complexity in 
  \[ \OT(d^3 B_L  + d^3 B_\beta + d \log(k) + d\log(l)) \]
  which simplifies to
  \[ \OT(d \bs(\mathfrak a) + d^2 \bs(\alpha) + d^3 \log(\lvert \Delta_K \rvert) + d^2 C) \]
  using the bound $B_L \in \OT(\min(\mathfrak b) + d^2 + \log(C_2) + \log(\lvert \Delta \rvert))$ derived from (\ref{eq:LLLbasis}).
  Since the complexity of the $\widetilde{\mathrm{L}}^1$-algorithm is in $\OT(d^3 \bs(\mathfrak a) + d^3 C)$ the claim follows.
\end{proof}

\begin{remark}\hfill
  \begin{enumerate}
    \item
      Note that the computation of the approximated reduced basis gives a big contribution to the overall complexity of Algorithm~\ref{alg:reduction}.
      It is therefore important to compute the approximated reduced basis only once, when reducing lots of elements of $K$ modulo a fixed ideal.
      More precisely the reduction of $n$ elements $\alpha_1,\dotsc,\alpha_n \in K$ can be done in
      \[ \tilde O(d^3 \bs( \mathfrak a) + nd  \bs(\mathfrak a) + n d^2 \max_{i} \bs(\alpha_i) + (n+d)d^2 C + nd^3 \log(\left|\Delta_K \right|)).   \]
    \item
      A reduced element is not necessarily of small size since the $T_2$-norm of a field element alone does not control the size of the element.
      More precisely if $\alpha$ is in $K$ and $k \in \Z_{>0}$ is the denominator of $\alpha$ then we have 
      \[ \bs(\alpha) = d \log(\left\| k \alpha\right\|_\infty) + \log(k) \leq (d+1) \log(k) + C + \log(\|\alpha\|). \]
      Thus in addition we need to control the size of the denominator to ensure that the reduced element is small with respect to $\bs$.
  \end{enumerate}
\end{remark}

\subsection*{Normalization.}

The normalization is the key difference between our approach and the one of Cohen~\cite{Cohen1996}.
It is the strategy that together with the reduction prevents the coefficient swell by calculating a pseudo-basis for which the ideals are integral with size bounded by invariants of the field.
The connection between the size of the integral coefficient ideals and denominators of the matrix entries is seen as follows.
Assume that $(A,(\mathfrak a_i)_i)$ is a pseudo-matrix of an $\mathcal O_K$-module $M$ and $A_i$ is the $i$-th row of $A$.
Since we consider only modules $M$ contained in $\mathcal O_K^n$ we see that $\mathfrak a_i A_i \subseteq \mathcal O_K^n$ allowing us to bound the denominators of the entries of $A_i$ by $\min(\mathfrak a_i)$.
\par
Since $\mathfrak a_i A_i = \alpha \mathfrak a_i (1/\alpha)A_i$ we can adjust our coefficient ideals by scalars from $K$ (while multiplying the row with the inverse). 
Therefore the task is to find an integral ideal $\mathfrak b$ such that $\mathfrak a \mathfrak b^{-1}$ is principal and $\inorm{\mathfrak b}$ is small.
Basically we just have to find a small integral representative of the ideal class of $\mathfrak a$.
The usual proof of the finiteness of the class number provides us with such a small representative and a norm bound involving Minkowski's constant.
As this is not suited for algorithmic purposes we handle this problem using $\Xi$-LLL reduced bases.
\par
We write $\mathfrak a = \mathfrak b /k$ and $\mathfrak b^{-1} = \mathfrak c /l$ with $k$ and $l$ the denominator of $\mathfrak a$ and $\mathfrak b^{-1}$ respectively.
Applying the $\widetilde{\mathrm{L}}^1$-algorithm we find an element $\alpha \in \mathfrak c$ satisfying
\begin{align}\label{eq:smallelement}
  \| \alpha \| \leq \ell^{(d-1)/2} \lvert \Delta_K \rvert^{1/(2d)} \inorm{\mathfrak c}^{1/d},
\end{align}
that is
\[ \inorm{\alpha} \leq \ell^{d^2} \sqrt{\lvert \Delta_K \rvert} \inorm{\mathfrak c} .\]
Then the ideal $\tilde{\mathfrak a}$ defined by $\tilde{\mathfrak a} = (\alpha/l) k\mathfrak a$ is integral since $\alpha \in \mathfrak c = l \mathfrak b^{-1} = l(k\mathfrak a)^{-1}$.
Moreover its norm satisfies
\[ \inorm{\tilde{\mathfrak a}} = \inorm{\alpha/l}\inorm{k\mathfrak a} = \frac{\inorm{\alpha}}{\inorm{\mathfrak c}} \leq \ell^{d^2} \sqrt{\lvert \Delta_K \rvert}. \]
and is therefore bounded by invariants of the field.

\begin{algorithm}[ht]
\caption{Normalization of a one-dimensional module}
\begin{algorithmic}[1]\label{alg:normalization}
\REQUIRE $A =(\alpha_1,\dotsc,\alpha_n) \in K^{n}$, fractional ideal $\mathfrak a$ of $K$ with denominator $k$.
\ENSURE  $\tilde A \in K^{n}$, $\tilde {\mathfrak a} \subseteq \OK$ such that $\inorm{\tilde{\mathfrak a}} \leq \ell^{{d^2}}\sqrt{|\Delta_K|}$ and $\mathfrak a  A = \tilde{\mathfrak a}\tilde A$.
\STATE Compute $\mathfrak b^{-1} = \mathfrak c /l$ where $\mathfrak b$ is the numerator of $\mathfrak a$.
\STATE Let $\alpha$ be the first element of an approximated $\Xi$-LLL reduced basis of $\mathfrak c$.
\RETURN $l/(k\alpha)A$, $(\alpha/l)k \mathfrak a$.
\end{algorithmic}
\end{algorithm}

\begin{proposition}\label{prop:normalization}
  Algorithm \ref{alg:normalization} is correct and its output satisfies
  \begin{align*}
    \bs(\tilde \alpha) &\in \tilde O\left( \bs(\ag) + \max_i \bs(\alpha_i) + d \log(\lvert \Delta_K \rvert) + dC\right),\\
    \bs(\tilde{\mathfrak a}) &\in \tilde O \left( d^4 + d^2\log(|\Delta_K|)\right),
  \end{align*}
  where $\tilde \alpha \in K$ is a coefficient of $\tilde A$. Its complexity is in
  \[ \tilde O(d(d^2 + n) \bs(\mathfrak a) + dn \max_i(\bs(\alpha_i)) + d^2(d+n) (\log (\lvert \Delta_K \rvert) + C)). \]
\end{proposition}

\begin{proof}
  The correctness of the algorithm follows from the preceding discussion.
  Computing the inverse of $\mathfrak b$ can be done in $\OT(d\bs(\mathfrak b) + d^3 \log (\lvert \Delta_K \rvert) + d^2 C)$.
  The output satisfies $\bs(\mathfrak c) \leq \bs(\mathfrak b)$ as well as $l \leq \min(\mathfrak b)$.
  The second step invokes the $\widetilde{\mathrm{L}}^1$-algorithm whose complexity is in $ \tilde O(d^3 \bs(\mathfrak c) + d^3 C)$  and which computes a small element $\alpha \in \mathfrak c$ with the property as in (\ref{eq:smallelement}).
  Now this bound on the $T_2$-norm translates into $\bs(\alpha) \in \OT(\bs(\mathfrak b)/d + \log(\lvert \Delta_K \rvert) + C)$.
  The element $\alpha/l$ can be computed in $\OT(\bs(\alpha) + \log(l))$ and satisfies $\bs(\alpha/l) \leq \bs(\alpha) + d \log(l)$.
  Thus computing the new coefficient ideal $(\alpha/l)k\mathfrak a = (\alpha/l)\mathfrak b$ costs $\OT(d^3 \bs(\alpha/l) + d \bs(\mathfrak b) + d^2 C) \subseteq \OT(d^2 \bs(\mathfrak a) + d^3 (\log(\lvert \Delta_K\rvert) + C)$.
  \par
  It remains to consider the multiplication of $A$ by $l/(k\alpha)$.
  Inverting $\alpha$ and multiplying $\alpha^{-1}$ by $l/k$ has complexity in $\OT(d^2 \bs(\alpha) + d^2 C + \log(k) + d \log(l))$.
  Since $\bs(l/(k\alpha)) \in \OT(d\bs(\alpha) + d \log(l) + d \log(k) + C)$ the multiplication with $A$ has complexity in 
  \[ \OT(d(d+n)(d \bs(\alpha) + d \log(l) + d \log(k)) + dn \max_{i}(\bs(\alpha_i)) + d(d+n)C), \]
  which reduces to $\OT(d(d+n)\bs(\mathfrak a) + dn \max_{i}(\bs(\alpha_i)) + d^2(d+n)(\log(\lvert \Delta_K \rvert) + C))$.
  Now the claim follows.
\end{proof}

%%%%%%%%%%%%%%%%%%%%%%%%%%%%%%%%%%%%%%%%%%%%%%%%%%%%%%%%%%%%%%%%%%%%%%%%%%%%%%%%
%
%  Computation of determinants over rings of integers
%
%%%%%%%%%%%%%%%%%%%%%%%%%%%%%%%%%%%%%%%%%%%%%%%%%%%%%%%%%%%%%%%%%%%%%%%%%%%%%%%%

\section{Computation of determinants over rings of integers}\label{sec:determinant}

As already mentioned the important step during our pseudo-HNF algorithm is the ability to reduce the entries modulo some integral ideal involving the determinantal ideal of the module.
The algorithms presented in this section describe how to obtain the determinantal ideal in case it is not known in advance.
We first describe a polynomial algorithm for computing the determinant of a square matrix over $\mathcal O_K$.
Already for integer matrices computing determinants is a rather involved task, see \cite{Kaltofen2004} for a survey of different approaches and their complexity.
Performing very well in practice and being a deterministic polynomial algorithm we present a determinant algorithm for matrices over $\mathcal O_K$ which is based on the small primes modular approach.

\subsection*{Bounding the size of the output.}

The underlying idea of a modular determinant algorithm is the possibility to bound the size of the result before the actual computation.
For a matrix $A = (a_{ij})_{i,j} \in \mathcal O_K^{n\times n}$ denote by $\lvert A \rvert $ the bound $\max_{i,j}\{ \left\| a_{ij}\right\|_\infty\}$.

\begin{lemma}\label{lem:bound}
  Let $A = (a_{ij})_{i,j} \in \mathcal O_K^{n\times n}$ Then $\left\| \det(A) \right\|_\infty \leq n^n C_1 C_2^n \lvert A \rvert^n$, that is,
  \[ \log(\left\| \det(A) \right\|_\infty) \in O(n \log(n\lvert A \rvert) + n C). \]
\end{lemma}

\begin{proof}
  We have $\det(A) = \sum_{\sigma \in \mathfrak S_n}\prod_{i=1}^n a_{i,\sigma(i)}$ and therefore
  \[ \left\| \det(A) \right\|_\infty \leq C_1 \left\| \det(A) \right\| \leq C_1 n! \max_{i,j}(\left\| a_{ij} \right\|)^n \leq C_1 C_2^n n^n \lvert A \rvert. \qedhere\] 
\end{proof}

Recall that in the absolute case $\mathcal O_K = \Z$ we get the same bound without the term $nC$ and we can immediately formulate a modular algorithm for determinant computations: Find $B \in \Z_{>0}$ such that $\lvert \det(A) \rvert < B/2$. Then compute the determinant of the matrix modulo $B$ and obtain a number $d \leq B/2$ such that $d \equiv \det(A) \bmod B$. Since $d$ and $\det(A)$ are bounded by $B/2$ they must be equal.
\par
Let us now show that the recovering technique can be applied to algebraic integers in place of rational integers.

\begin{lemma}\label{lem:recover}
  Let $\alpha = \sum_{i=1}^d a_i \omega_i$ and $\beta = \sum_{i=1}^d b_i \omega_i$ be two algebraic integers in $\mathcal O_K$. Assume there exists $B \in \R_{>0}$ such that $\lvert a_i \rvert, \lvert b_j \rvert < B/2$ for all $1 \leq i,j \leq d$ and $\alpha \equiv \beta \mod (B)$. Then $\alpha = \beta$.
\end{lemma}

\begin{proof}
  Since $(\omega_i)_{i}$ is a $\Z$-basis of $\mathcal O$, the family $(B \omega_i)_i$ is a $\Z$-basis of the principal ideal $(B)$.
  Hence $\alpha \equiv \beta \mod{(B)}$ is equivalent to the divisibility of $a_i - b_i$ by $B$ for all $1 \leq i \leq d$.
  Using the coefficient bound we obtain
  \[ 0 \leq \lvert a_i - b_i \rvert \leq \lvert a_i \rvert + \lvert b_i \rvert < B. \]
  We conclude that $a_i = b_i$ for all $1 \leq i \leq d$, that is, $\alpha = \beta$.
\end{proof}

We can now proceed as in the integer case.
After computing the determinant modulo several primes $p$ we combine the results via the Chinese remainder theorem.
As soon as the product exceeds the a priori bound from Lemma~\ref{lem:bound} we can recover the actual value using Lemma~\ref{lem:recover}.
As $\mathcal O_K/ (p)$ is in general not a nice ring to work with, we want to decompose $(p)$ into prime ideals $\mathfrak p$ of $\mathcal O_K$ allowing for computations in the finite field $\mathcal O_K / \mathfrak p$. 
Again the result modulo $(p)$ can be obtained invoking the Chinese remainder theorem.
We address the computational complexity of this two stage Chinese remaindering in the following section.

\subsection*{Chinese remaindering for rational primes and prime ideals over ring of integers.}
Let $p \in \Z_{>0}$ be rational prime.
By the theory of maximal orders in number fields, see \cite[4.6.2]{Cohen1993}, there exists a factorization 
\[ (p) = \prod_{i=1}^g \mathfrak p_i^{e_i} \]
into pairwise different prime ideals $\mathfrak p_i \subseteq \mathcal O_K$ with exponents $e_i \in \Z_{>0}$.
Moreover there exists $f_i \in \Z_{>0}$ such that $\dim_{\F_p}\mathcal O_K / \mathfrak p_i = f_i$ and
\[ \sum_{i=1}^g e_i f_i = d. \]
Note that we are only interested in \textit{unramified} primes $p$ where all $e_i$'s are equal to $1$, or else we would have to compute the determinant over $\mathcal O_K / \mathfrak p^{e_i}$ a ring containing zero divisors.
In this unramifed case we also have $\sum_{i=1}^g f_i = d$.
By another famous theorem from algebraic number theory, see \cite[Theorem 4.8.8.]{Cohen1993}, we know that the primes not dividing $\Delta_K$ are exactly the unramified primes.\par
On the other hand we need to restrict ourselves to rational primes $p$ not dividing $[ \mathcal O_K \colon \Z[\alpha]]$ since then we can efficiently compute the decomposition and the residue fields.
This is due to the following theorem of Dedekind-Kummer, see \cite[Theorem 4.8.13.]{Cohen1993}.
Recall that $f$ is the defining polynomial of the number field $K$ chosen as in our assumptions.

\begin{proposition}\label{prop:dec}
  Let $p$ be a rational prime not dividing $[ \mathcal O \colon \Z[\alpha]]$ and $\overline f = \prod_{i=1}^g \overline f_i^{e_i}$ the factorization of $\overline f \in \F_p[X]$ into irreducible polynomials. Then
      \[ \mathcal O_K / p \mathcal O_K \cong \Z[\alpha]/p\Z[\alpha] \cong \F_p[X]/(\overline f) \cong \prod_{i=1}^g \F_p[X]/(\overline f_i^{e_i}). \]
\end{proposition}

Computing the factorization of $(p)$ in $\mathcal O_K$ is therefore equivalent to a polynomial factorization over a finite field.
We now describe the complexity of passing to the residue field and of working in it.
Assume that $p$ is a fixed rational prime, unramified and not dividing $[ \mathcal O_K \colon \Z[\alpha]]$.
The first task is the factorization of $f$ modulo $p$ which can be achieved by the deterministic algorithm of Shoup \cite[Theorem 3.1.]{Shoup1990}.

\begin{proposition}
  Let $p \in \Z_{>0}$ be a rational prime. 
  The number of $\F_p$ operations needed to compute the factorization of $\overline f \in \F_p[X]$ into irreducible polynomials is in $\OT(p^{1/2}\log(p)^2 d^{2})$.
  Thus this has complexity in $\OT(p^{1/2}d^{2} \log( p)^{3})$.
\end{proposition}

For each irreducible factor $\overline f_i \in \F_p[X]$ of $\overline f$ we obtain the diagram
\[ \begin{CD} \mathcal O_K @>>> \mathcal O_K/(p) @>\pi>> \F_p[X]/(\overline f) @>\pi_i>>  \F_p[X]/(\overline f_i) \end{CD},\]
where $\pi$ and $\pi_i$ are the corresponding projections.
We now determine the complexity of passing from $\mathcal O_K$ to $\F_p[X]/(\overline f_i)$.
Let $\beta = \sum_{i=1}^d b_i \omega_i$ be an integral element.
Since $\pi$ is a ring homomorphism we obtain 
\[ \pi(\alpha) = \sum_{j=1}^d \overline {a_i} \pi(\omega_j) \]
where $\overline{\phantom{x}}$ denotes reduction $\Z \to \F_p$.
Therefore we (only) need to evaluate $\pi$ on the integral basis $(\omega_j)_{j}$. 
Denote by $\alpha$ the primitive element of $K$ chosen in our assumption with minimal polynomial $f$.
We consider the transformation matrix $M = (m_{ij})_{i,j} \in \Z^{d\times d}$ between the power basis $(\alpha^j)_{1 \leq j \leq d}$ and the integral basis $\Omega$, which is defined by the equations
\[ \alpha^j = \sum_{i=1}^d m_{ij} \omega_i \]
for $1 \leq j \leq d$.
Then $\pi(\omega_i)$ is just the $i$-th column of $\overline M^{-1} \in \F_p^{d\times d}$, where $\overline M \in \F_p^{d\times d}$ is the matrix obtained by reducing each entry of $M$ modulo $p$.
For the complexity analysis we need a bound on the size of $M$.
As $\alpha = \sum_{i=1}^d \varepsilon_i \omega_i$ with $\varepsilon_i \in \{0,1\}$ we have $\bs(\alpha) = d$ and therefore $\bs(\alpha^j) \leq j \bs(\alpha) + j C \leq d \bs(\alpha) + d C$, that is, $\log(\left\|\alpha^j \right\|_\infty) \leq d + C$.
This implies $\log(\lvert M \rvert) \leq C + d$ for the size of the entries of $M$.

\begin{proposition}\label{prop:red}
Let $\overline f_1,\dotsc,\overline f_g$, $\pi$ and $\pi_i$ and $\beta$ as in the preceding discussion. 
\begin{enumerate}
\item
  The $d\cdot g$ many images $\pi_i(\omega_j)$, $1 \leq j \leq d$, $1 \leq i \leq g$, can be computed with complexity in $\tilde O(d^3\log(p) + d^2 C)$.% $\tilde O(d^3 \log(p) + d^2 C)$.
\item
  Let $P$ be a set of primes. The reduction of the coefficient vector of $\beta$ modulo all primes in $P$ costs $\OT(d \sum_{l\in P} \log(l)) + \bs(\beta))$.
\item
  Assuming that $\pi_i(\omega_j)$, $1 \leq j \leq d$, $1 \leq i \leq g$ as well as the reduction of the coefficient vector of $\beta$ modulo $p$ is known, the computation of $\pi_i(\beta)$, $1 \leq i \leq g$, has complexity in $\tilde O(d^2 \log(p))$.
\end{enumerate}
\end{proposition}

\begin{proof}
(1): The reduction of $M$ modulo $p$ has complexity in $\tilde O(d^2 (\log(p) + C))$ and inverting the reduced matrix over the finite field $\mathbb F_p$ has complexity in $\tilde O(d^3(\log(p))$.
Then for each $1 \leq j \leq d$ we have to reduce the elements $\pi_i(\omega_j)$ modulo $\overline f_i$ for $1 \leq i \leq g$.
By \cite[Lemma 3.2]{Shoup1990} this has complexity in $d \tilde O(d \log(g)) \subseteq \tilde O(d^2)$. \par
(2): Using the remainder tree of Bernstein \cite[18.6]{Bernstein2008} the computation for each coefficient is in $\OT(\sum_{l \in P} \log(l) + \log(\|\beta\|_\infty))$.
\par
(3): We just have to compute $d$ products $\overline a_j \pi_i(\omega_j)$, $1 \leq j \leq d$ and $d$ additions of elements in $\mathbb F_p[X]/(\overline f_i)$. 
The last two steps have complexity in $\tilde O(d \deg(\overline f_i) \log(p))$.
Thus summing over all $1 \leq i \leq g$ we obtain a complexity in $\tilde O(d^2 \log(p))$.
\end{proof}

Working in the residue fields is just polynomial arithmetic over $\F_p$. For the sake of completeness we recall the necessary complexity, see for example \cite[Lemma 3.2]{Shoup1990}.
\begin{remark}
  \begin{enumerate}
\item
  Let $a,b \in \F_p$ and $\star \in \{+,-,\cdot,\div\}$. The complexity of $a \star b$ is in $\tilde O (\log( p))$
\item
  Multiplication of two polynomials of degree $\leq d$ in $\F_p[X]$ can be performed using $\tilde O(d)$ operations in $\F_p$.
\item
  Let $f,g \in \F_p[X]$ be two polynomials of degree $\leq d$. Then $f \mod g$ can be computed using $\tilde O(d)$ operations in $\F_p$. The greatest common divisor of $f$ and $g$ can be computed using $\tilde O(d)$ operations in $\F_p$.
\item
  Let $h \in \F_p[X]$ be a polynomial of degree bounded by $d$. 
  Assume we have $\overline g,\overline f \in \F_p[X]/(h)$ and $\star \in \{+,-,\cdot,\div\}$. 
  Then $\overline {g \star f}$ can be computed using $\tilde O(d)$ operations in $\F_p$. 
  Therefore each operation in $\F_p[X]/(h)$ has complexity in $\tilde O(d \log(p))$.
\end{enumerate}
\end{remark}

Finally we describe how to combine the computations in the finite fields to obtain a result in $\mathcal O_K/(N)$.
\par
Assume that we have a set $P$ of rational primes and $N = \prod_{p \in P} p$.
For each prime $p$ we have a factorization of $f$ modulo $p$ into irreducible factors $\overline f_i \in \F_p[X]$, $1 \leq i \leq g$.
Using the Chinese remainder theorem for polynomials we can construct a preimage under the map
\[ \F_p[X]/(\overline f) \to \prod_{i=1}^g \F_p[X]/(\overline f_i). \]
The next step is an application of the Chinese remainder theorem for rational integers for each coefficient yielding a preimage under the map
\[ \Z[X]/(\overline f,N) \longrightarrow \prod_{p \in P} \F_p[X]/(\overline f). \]
Finally we have to compute a preimage under the map $\ \mathcal O_K/(N) \to \Z[X]/(\overline f,N)$.

\begin{proposition}\label{prop:crt}
  Using the notation from the preceding paragraph the following holds:
  \begin{enumerate}
    \item
      Let $\overline h_i \in \F_p[X]/(\overline f_i)$ for $1 \leq i \leq d$. Computing $\overline h \in \F_p[X]/(\overline f)$ such that $\pi_i(\overline h) = \overline f_i $, $1\leq i \leq g$ has complexity in $\tilde O(d \log(p))$.
    \item
      Assume we are given $\overline g_p\in \F_p[X]/(\overline f)$ for $p \in P$. Then we can compute $\overline h \in \Z[X]/(f,N)$ with $\overline h = \overline g_p$ in $\mathbb F_p[X]/(\overline f)$ for $p \in P$ with complexity in $\tilde O(d \log(B) r)$ where $B \in \R_{\geq 0}$ is such that $p \leq B$ for all $p \in P$ and $\lvert P\rvert = r \geq 2$ is the number of involved primes.
    \item
      Given $\overline g \in \Z[X]/(f,N)$ the computation of a preimage under the map $\mathcal O_K / (N) \to \Z[X]/(f,N)$ has complexity in $\tilde O(d^2 (d+C+\log(N)))$.
  \end{enumerate}
\end{proposition}

\begin{proof}
  (1): This is Corollary 10.23 in \cite{Gathen2003}.\par
  (2): Due to Bernstein \cite[\S23]{Bernstein2008} Chinese remaindering involving $r$ moduli of size bounded by $B$ has complexity in $\tilde O(\log(B)r)$.
  Since we have $d$ coefficients, the result follows.\par
  (3): This is just a matrix vector product between the coefficients of $g$ and $M$.
\end{proof}

\subsection*{On the number, choice and size of primes.}
We still need to describe how many primes we need and of which size they are. 
By Theorem~\ref{lem:bound} the number $B = n^n C_1 C_2^n  \lvert A \rvert^n \in \R_{>0}$ satisfies $\left\|\det(A)\right\|_\infty \leq B$.
Choosing the first $r' = \lceil \log(B) \rceil$ primes we obtain $\prod_{i=1}^{r'} p_i > B/2$.
As we have seen there is a finite number of bad primes we need to avoid.
More precisely we are only interested in primes not dividing $\Delta_K$ and $[ \mathcal O_K \colon \Z[\alpha]]$.
As $\Delta_K$ and $[\mathcal O_K \colon \Z[\alpha]]$ have at most $\log(\lvert \Delta_K \rvert) + \log([\mathcal O_K \colon \Z[\alpha]])$ prime factors we see that the set of first $r = \log(B)+\log(\lvert \Delta_K \rvert ) + \log(\lvert \operatorname{disc}(f)\rvert))$ primes $P'$ contains a subset $P$ such that $\prod_{p \in P} p > B$ and no element of $P$ divides $\Delta_K$ or $[\mathcal O_K \colon \Z[\alpha]]$.
Here we have used that $[\mathcal O_K \colon \Z[\alpha]]$ divides $\lvert \operatorname{disc}(f)\vert$.
\par
We have the following classical result about the computation and size of the first $r$ primes.
It is an application of the detailed analysis of Rosser and Schoenfeld \cite{Rosser1962} as well as the sieve of Eratosthenes and can be found in \cite[Theorem 18.10]{Gathen2003}.

\begin{proposition}\label{prop:primes}
  Let $r \in \Z_{>0}$.
  The first $r$ prime numbers $p_1,\dotsc,p_r \in \Z_{>0}$ can be computed with complexity in $O(r(\log(r))^2 \log\log(r))$ and if $r \geq 2$ each prime satisfies $p_i \leq 2 r \ln(r)$, that is, $\log(p_i) \in \tilde O(\log(r))$.
\end{proposition}

\subsection*{The algorithm and its complexity.}
In the preceding sections we have collected all the tools that we need to describe and analyze the determinant computation.

\begin{algorithm}[ht]
\caption{Determinant computation over $\mathcal O_K$}
\begin{algorithmic}[1]\label{alg:det}
\REQUIRE $A\in \OK^{n\times n}$, 
\ENSURE $\det(A)$.
\STATE Compute a bound $B \in \R_{>0}$ on $\left\| \det(A) \right\|_\infty$ and set $r = \lceil\log(B) + \log(\Delta_K) + \log(\lvert \operatorname{disc}(f)\vert )\rceil \in \Z_{>0}$.
\STATE Compute the first $r$ primes and choose $r'$ many $P = \{p_1,\dotsc,p_{r'}\}$ not dividing $\Delta_K$ and $\lvert \operatorname{disc}(f)\rvert$.
\FOR {$p \in P$}
  \STATE Compute $\overline f_1,\dotsc,\overline f_g \in \F_p[X]$ such that  $\overline f = \overline f_1 \dotsb \overline f_g$.
  \STATE Compute $\pi_j(\omega_i)$ for $1 \leq i \leq d$ and $1 \leq j \leq g$.
  \STATE Compute $\pi_j(A) \in (\F_p[X]/(\overline f_j))^{n \times n}$ for $1 \leq j \leq g$.
  \STATE Compute $d_j = \det(\pi(A)) \in \F_p[X]/(\overline f_j)$ for $1 \leq j \leq g$.
  \STATE Compute $\overline g_p \in \F_p[X]/(f)$ such that $g_p = d_j$ in $\F_p[X]/(\overline f_j)$ for $1 \leq j \leq g$.
\ENDFOR
\STATE Compute $\overline g \in \Z[X]/(f,N)$ such that $\overline g = g_p$ in $\mathbb F_p[X]/(\overline f)$ for all $p \in P$.
\STATE Compute the image of $\overline g$ under $\Z[X]/(f,N) \to \mathcal O_K / (N)$ where $N = \prod_{p \in P}p$.
\RETURN $X$.
\end{algorithmic}
\end{algorithm}

\begin{theorem}
  Algorithm \ref{alg:det} is correct and has complexity in 
  \[ \tilde O(d^2r^{3/2} + r( d^2 C + d^3 + d n^3 + d^2 n^2)),\]
  where $r = n \log(\lvert A \rvert) + \log(\lvert \Delta_K \rvert) + nC$.
\end{theorem}

\begin{proof}
  The correctness follows from the preceding paragraphs. 
  By Proposition~\ref{prop:primes} Step~2 costs $\OT(r)$ and every $p \in P$ satisfies $p \leq 2 r \ln(r)$.
  As $\log(p) \in \OT(\log(r)) = \OT(1)$ we will ignore all polynomial terms in $\log(p)$.
  Let us now consider the loop in Steps~3--9 excluding Step~5.
  As already noticed the factorization of $f$ modulo $p$ has complexity in $\tilde O(p^{1/2} d^2 + dC) \subseteq \tilde O( r^{1/2} d^2 + d C)$.
  By Proposition~\ref{prop:red} computing the image of $(\omega_i)_i$ under the various $\pi_j$ has complexity in $\tilde O(d^3 + d^2 C)$.
  Each determinant computation consists of $O(n^3)$ operations in $\mathbb F_p[X]/(\overline f_i)$ taking $\tilde O(n^3 \deg(\overline f_i))$ bit operations in total.
  Consequently by summing over all $1 \leq i \leq g$ we see that Step~7 has complexity in $\tilde O(d n^3 )$.
  By virtue of Proposition~\ref{prop:crt} Step~8 has complexity in $\tilde O(d)$.
  Since these steps are repeated $r$ times we obtain a complexity in $\OT(r(d^2 r^{1/2} + d^2 C + d^3 + d n^3))$.
  Now consider the missing Step~5.
  Reducing the coefficients of all entries of $A$ modulo all primes $p \in P$ has complexity in $\OT( d n^2 \sum_{p \in P} \log(p) + d n^2 \log(\lvert A \rvert)) \subseteq \OT( d n^2 r + d n^2 \log(\lvert A \rvert))$ by Proposition~\ref{prop:red}~(2).
  To compute $\pi_j(A)$ we apply item (3) of the same proposition and arrive at a complexity of $\OT(d^2 n^2 r))$ since we have to do it $r$ times.
  In total the inner loop in Steps~3--9 has a complexity in 
  \begin{align*} \OT(r(d^2 r^{1/2} + d^2 C + d^3 + d n^3 + d^2 n^2) + dn^2 \log(\lvert A \rvert)).
  \end{align*}
  As Steps~10 and 11 have complexity in $\tilde O(d r)$ and $\tilde O(d^2(C+d+\log(N)) \subseteq \tilde O(d^2(C + d+ r))$ respectively, we get an overall complexity in 
  \[ \tilde O(r(d^2 r^{1/2} + d^2 C + d^3 + d n^3 + d^2 n^2) + d n^2 \log(\lvert A \rvert) + d^2 C + d^3 + d^2 r). \]
  Finally we use the fact that  $r \in O(\log(B) + \log(\lvert \operatorname{disc}(f)\rvert) + \log(\lvert \Delta_K \rvert)) \subseteq \tilde O(n \log(\lvert A \rvert) + nC + \log( \lvert \Delta_K \rvert))$ to conclude that the complexity of Algorithm~\ref{alg:det} is in
  $\tilde O(d^2r^{3/2} + r( d^2 C + d^3 + d n^3 + d^2 n^2))$.
\end{proof}

In case we fix the number field $K$, that is, we ignore the constants coming from field arithmetic, the complexity of Algorithm~\ref{alg:det} reduces to $\tilde O((n\log(\lvert A \rvert))^{3/2} + n^4 \log(\lvert A \rvert))$.
\par
Note that in contrast to the integer case our algorithm is not softly linear in $\log(\lvert A \rvert)$ which can be explained as follows:
Recall that our small primes approach needs at least $\log(\lvert A \rvert)$ primes which are roughly of the same order as $\log( \lvert A \rvert)$.
As the deterministic factorization in $\F_p$ has costs $\OT(p^{1/2})$ (ignoring the dependency on the degree), the complexity of all factorizations contains at least a factor of $\log(\lvert A \rvert) \log(\lvert A \rvert)^{1/2} = \log(\lvert A \rvert)^{3/2}$.
Consequently we see that the exponential factorization algorithm is the bottleneck of our determinant algorithm.
While there exist various probabilistic polynomial algorithms for the factorization over $\F_p$, they are unusable for us, since we are aiming at an deterministic polynomial pseudo-HNF algorithm.
We can now address the problem of computing the determinantal ideal.

\begin{corollary}\label{cor:det}
  Assume that $M = (A,(\mathfrak a_i)_i)$ is a pseudo-matrix with $A \in \mathcal O_K^{n \times n}$. There exists a deterministic algorithm computing
  \[ \mathfrak d(M) = \det(A) \prod_{i=1}^n \mathfrak a_i \]
  with complexity in
  \[ \tilde O(d^2r^{3/2} + r( d^2 C + d^3 + d n^3 + d^2 n^2) + d^2 n B + d^4 n C),\]
  where $r = n \log(\lvert A \rvert) + \log(\lvert \Delta_K\rvert) + nC$ and $B= \max_i \bs(\mathfrak a_i)$.
\end{corollary}

\begin{proof}
  It remains to evaluate the complexity of the ideal product.
  As a divide and conquer approach shows that the product can be computed with complexity in $\tilde O(d^2 n \log(n)  B + d^4 n C)$ the claim follows.
\end{proof}

\subsection*{The rectangular case.}

The case where the pseudo-matrix is not square is more involved.
We will now describe an algorithm for computing an integral multiple of the determinantal ideal of a pseudo-matrix.

\begin{theorem}\label{thm:detrect}
There exists a deterministic algorithm that given a matrix $A \in \mathcal O_K^{n \times m}$, $m \leq n$, returns the rank $s$ of $A$, a non-singular $s\times s$ submatrix $A'$ of $A$ and $\det(A')$.
The algorithm has complexity in
\[ \OT(d^2 r^{3/2} + r(d^2C + d^3 + dnm^2 + d^2 nm)), \]
where $r = m \log(\lvert A \rvert) + \log(\lvert \Delta_K \rvert) + mC$.
\end{theorem}

\begin{proof}
  Let $s$ be the rank of $A$ and $A' \in \mathcal O_K^{s \times s}$ a non-singular submatrix of $A$.
  Using Lemma~\ref{lem:bound} we see that $\log(\|\det(A')\|_\infty) \in O(s \log(s \lvert A' \rvert) + sC) \subseteq O(m \log(m\lvert A \rvert) + mC)$. 
  Now let $B \in \R_{>0}$ be a number with $\log(B) \in O(m \log(m\lvert A \rvert) + m C)$ such that $B$ exceeds $2\left\|\det(A')\right\|_\infty$ for all non-singular $s\times s$ submatrices of $A$, and let $\mathfrak p_1,\dotsc,\mathfrak p_l$ be coprime prime ideals of $\mathcal O_K$ such that $B \in \mathfrak p_1 \mathfrak p_2 \dotsb \mathfrak p_l$.
  If now $A'$ is a non-singular $s\times s$ submatrix of $A$ there exists $i$ such that $A' \bmod \mathfrak p_i$ has non-zero determinant and consequently $A \bmod \mathfrak p_i$ has rank $s$.
  For if this is not the case, we would have $\det(A') \equiv 0 \bmod \mathfrak p_1\dotsb\mathfrak p_l$ yielding $\det(A') \equiv 0 \bmod (B)$ and $\det(A') = 0$ by Lemma~\ref{lem:recover}, a contradiction.
  \par
  Thus to find the rank we choose a set of primes $P$ such that $\prod_{p \in P} p \geq B$ and all $p \in P$ are unramified and do not divide $[ \mathcal O_K : \Z[\alpha]]$.
  For every prime $p \in P$ we compute the prime ideal factorization of $p\mathcal O_K$ and for all these prime ideals we compute the rank of $A \bmod \mathfrak p$ using Gaussian elimination.
  Similar to Algorithm~\ref{alg:det} this has complexity in $\OT(d^2 r^{3/2} + r(d^2C + d^3 + dnm^2 + d^2 nm))$, 
  where $r = m \log(\lvert A \rvert) + \log(\lvert \Delta_K \rvert) + mC$.
  \par
  Let $\mathfrak p$ be a prime ideal lying above one of the $p \in P$ such that $A \bmod \mathfrak p$ has maximal rank.
  Then the rank of $A \bmod \mathfrak p$ is the rank $s$ of $A$ and we can find a non-singular $s\times s$ submatrix $A'$ of $A$.
  The computation of $\det(A')$ using Algorithm~\ref{alg:det} has complexity in
  \[ \OT(d^2 r^{3/2} + r(d^2C + d^3 + ds^3 + d^2 s^2)) \subseteq \OT(d^2 r^{3/2} + r(d^2C + d^3 + dnm^2 + d^2 nm)). \]
\end{proof}

Combining this result with Corollary~\ref{cor:det} immediately yields:

\begin{corollary}
  Assume that $M = (A,(\mathfrak a_i)_i)$ is a pseudo-matrix with $A \in \mathcal O_K^{n \times m}$, $n \geq m$, of rank $m$. There exists a deterministic algorithm computing a multiple of $\mathfrak d(M)$
  with complexity in
  \[ \tilde O(d^2r^{3/2} + r( d^2 C + d^3 + d nm^2 + d^2 nm) + d^2 m B + d^4 m C),\]
  where $r = m \log(\lvert A \rvert) + \log(\lvert \Delta_K\rvert) + mC$ and $B= \max_i \bs(\mathfrak a_i)$.
\end{corollary}

%%%%%%%%%%%%%%%%%%%%%%%%%%%%%%%%%%%%%%%%%%%%%%%%%%%%%%%%%%%%%%%%%%%%%%%%%%%%%%%%
%
%  The pseudo-HNF algorithm
%
%%%%%%%%%%%%%%%%%%%%%%%%%%%%%%%%%%%%%%%%%%%%%%%%%%%%%%%%%%%%%%%%%%%%%%%%%%%%%%%%

\section{The pseudo-HNF algorithm}

\subsection*{Constructing idempotents.} 

In order to compute the pseudo-HNF over Dedekind domains, we use the constructive version of the Chinese remainder theorem introduced by Cohen in \cite{Cohen1996}.
More precisely given coprime integral ideals $\mathfrak a$ and $\mathfrak b$ of $\mathcal O_K$, we need to find $\alpha \in \mathcal O_K$ such that $\alpha \in \mathfrak a$ and $1 - \alpha \in \mathfrak b$.
This problem is closely connected to the computation of the sum of $\mathfrak a$ and $\mathfrak b$:
The HNF of the matrix
\[ A = \left( \begin{array}{c|c} M_\mathfrak a & M_\mathfrak a \\ \hline \mathbf{0} & M_\mathfrak b \end{array}\right) \]
is equal to
\[ \left( \begin{array}{c|c} \ast & \mathbf{0} \\ \hline U & M_{\mathfrak a + \mathfrak b} \end{array} \right) = \left( \begin{array}{c|c} \ast & \mathbf{0} \\ \hline U & \mathbf{1}_d \end{array} \right)\]
for some $U \in \Z^{d\times d}$ since $\mathfrak a + \mathfrak b = \mathcal O_K$.
Denoting by $v \in \Z^d$ be the first row of $U$ we see that the element $\alpha = \sum_{i=1}^d v_i \omega_i$ of $\mathcal O_K$ satisfies $\alpha \in \mathfrak a$ and $1 - \alpha \in \mathfrak b$.
%If $\gamma$ is any element of $\mathfrak a \mathfrak b$, then $\alpha - \gamma \in \mathfrak a$ and $1 - (\alpha - \gamma) = 1 - \alpha + \gamma \in \mathfrak b$.
%Therefore, as in the absolute case, the element $\alpha$ can be adjusted using elements from the product $\mathfrak a \mathfrak b$.
%In particular we can apply Algorithm~\ref{alg:reduction} to $\alpha$ and $\mathfrak a \mathfrak b$ to ensure smallness of $\alpha$ and $1-\alpha$ with respect to the size of $\mathfrak a \mathfrak b$.

\begin{lemma}\label{alg:crt}
  Given coprime integral ideals $\mathfrak a$, $\mathfrak b$ of $\mathcal O_K$, 
there exists a deterministic algorithm which computes elements $\alpha \in \mathfrak a$ and $\beta \in \mathfrak b$ such that $\alpha + \beta = 1$.
  %the product $\mathfrak a \mathfrak b$ and an element $\alpha \in \mathfrak a$ such that $1 - \alpha \in \mathfrak b$.
  Moreover the output satisfies $\bs(\alpha),\bs(\beta) \in \OT((\bs(\mathfrak a)+\bs(\mathfrak b))/d)$ and the complexity of the algorithm is in 
  \[  \OT(d (\bs(\mathfrak a)+\bs(\mathfrak b))). \]
\end{lemma}

\begin{proof}
  We use the same notation as in the preceding discussion.
  Note that $\lambda = \min(\mathfrak a)\min(\mathfrak b)$ satisfies $\lambda \Z^{2d} \in [A]_\Z$ allowing us to compute the HNF with complexity in $\OT(d^3 \log(\min(\mathfrak a)\min(\mathfrak b))) = \OT(d(\bs(\mathfrak a)+ \bs(\mathfrak b))$.
  Moreover as $\log(\lvert U \rvert) \leq 2 \log(\lambda)$ we know that $\bs(\alpha) = d \log(\lvert v \rvert) \in O((\bs(\mathfrak a) + \bs(\mathfrak b))/d)$.
  \end{proof}

Using this construction we can now describe an algorithm, which plays the same role as the extended GCD algorithm over the integers.
It is the workhorse of the pseudo-HNF algorithm and is accompanied by the normalization and reduction procedures which provide bounded input.

\begin{algorithm}[ht]
  \caption{Euclidean step}
  \begin{algorithmic}[1]\label{alg:crt2}
    \REQUIRE Fractional ideals $\mathfrak a$, $\mathfrak b$ and elements $\alpha,\beta \in K$.
    \ENSURE $\mathfrak g = \alpha \mathfrak a + \beta \mathfrak b$, $\mathfrak g^{-1}$ and $\gamma \in \mathfrak a\mathfrak g^{-1}$ and $\delta \in \mathfrak b \mathfrak g^{-1}$ such that $\alpha \gamma + \beta \delta = 1$.
    \STATE Compute $\mathfrak g = \alpha \mathfrak a + \beta \mathfrak b$, $\mathfrak g^{-1}$, $\mathfrak a\mathfrak g^{-1}$ and $\mathfrak b \mathfrak g^{-1}$.
    \STATE Apply Lemma~\ref{alg:crt} to $\alpha \mathfrak a \mathfrak g^{-1}$ and $\beta \mathfrak b \mathfrak g^{-1}$ and denote the output by $\tilde \gamma$, $\tilde \delta$.
    \RETURN  $\gamma = \tilde \gamma \alpha^{-1}$ and $\delta = \tilde \delta \beta^{-1}$.
  \end{algorithmic}
\end{algorithm}

\begin{proposition}
  Algorithm \ref{alg:crt2} is correct and has complexity in
  \[ \tilde O(d^2(\bs(\mathfrak a)+\bs(\mathfrak b)) + d^4 (\bs(\alpha)+\bs(\beta)) + d^3 C + d^3 \log (\lvert \Delta_K \rvert)). \]
  The output satisfies $\bs(\gamma),\bs(\delta) \in \OT( (\bs(\mathfrak a) + \bs(\mathfrak b))/d + d(\bs(\alpha)+\bs(\beta)) + C)$.
\end{proposition}

\begin{proof}
  Correctness is clear.
  The first step consists of the computation of $\alpha\mathfrak a$ and $\beta \mathfrak b$, which has complexity in $\OT(d^3 (\bs(\alpha)+\bs(\beta)) + d (\bs(\mathfrak a)+\bs(\mathfrak b)) + d^2 C)$.
  Denote by $B$ the value $\bs(\alpha\mathfrak a) + \bs(\beta\mathfrak b) \in O(\bs(\mathfrak a) + \bs(\mathfrak b) + d^2 \bs(\alpha) + d^2 \bs(\beta) + dC)$.
  While the computation of $\mathfrak g$ has complexity in $\OT(dB)$ the inversion costs $\OT(dB + d^3 \log(\lvert \Delta_K \rvert) + d^3 C)$.
  As $\bs(\mathfrak g) \in O(B)$ the inverse ideal $\mathfrak g^{-1}$ also satisfies $\bs(\mathfrak g^{-1}) \in O(B)$.
  Finding the product $\beta \mathfrak b \mathfrak g^{-1}$ and $\alpha \mathfrak a \mathfrak g^{-1}$ then has complexity in $\OT(d^2 B + d^3 C)$ and the size of both integral ideals is in $\OT(B)$.
  Hence invoking Lemma~\ref{alg:crt} has a complexity in $\OT(dB)$ and the resulting elements satisfy $\bs(\tilde\gamma),\bs(\tilde\delta) \in \OT(B/d)$.
  Finally we have to compute inverses and products.
  While $\alpha^{-1}$ and $\beta^{-1}$ can be computed in $\OT(d^2 \bs(\alpha) + d^2 \bs(\beta) + d^2C)$ the costs of the products are in $\OT(d^2 \bs(\tilde \gamma) + d^2 \bs(\tilde \delta) + d^2 \bs(\alpha) + d^2 \bs(\beta) + d^2 C)$.
  Thus the ideal product dominates the complexity of the algorithm and the claim follows.
  Note that $\bs(\gamma) = \bs(\tilde\gamma \alpha^{-1}) \leq \bs(\tilde\gamma) + d \bs(\alpha) + C \in \OT(B/d + d(\bs(\alpha))$ and a similar result holds for $\bs(\delta)$.
  \end{proof}

\subsection*{The main algorithm and its complexity.}
Assume that $M \subseteq \mathcal O_K^m$ is an $\mathcal O_K$-module of rank $m$ given by a pseudo-basis $(A,(\mathfrak a_i)_i)$ with $A \in K^{n\times m}$ ($n \geq m$).
Using this input we now describe a polynomial time algorithm for computing the pseudo-HNF of $M$.
The algorithm is a variant of the so-called modular algorithm of Cohen, the big difference being the normalization of the coefficient ideals.
Using this extra feature we are able bound the denominators of the coefficients of the matrix.
Together with the reduction procedure this will allow us to prove polynomial running time.
Invoking Theorem~\ref{thm:detrect} we may assume that we know the determinantal ideal $\mathfrak d$ of $M$.
This case often occurs, for example when computing with ideals in relative extension.

\begin{algorithm}[ht]
  \caption{pseudo-HNF of a full-rank pseudo-matrix}
  \begin{algorithmic}[1]\label{alg:pseudohnf}
    \REQUIRE Full-rank pseudo-matrix $(A,(\mathfrak a_i)_i)$ of an $\mathcal O_K$-module $M \subseteq \mathcal O_K^m$ and $\mathfrak d = \det(M)$.
    \ENSURE Pseudo-HNF $(B,(\mathfrak b_i))$ of $M$. 
    \STATE Let $(B,(\mathfrak b_i)_i) = (A,(\mathfrak a_i)_i)$. Normalize $(A_i,\mathfrak a_i)_{1 \leq i \leq n}$ with Algorithm~\ref{alg:normalization}.
    \STATE Reduce $A_i$ modulo $\mathfrak d \mathfrak a_i^{-1}$ using Algorithm~\ref{alg:reduction} for $1 \leq i \leq n$.
    \STATE $\mathfrak D \leftarrow \mathfrak d$.
    \FOR{$i = n,\dotsc, n - m + 1$}
      \FOR{$j = i-1,\dotsc,1$}
      \STATE $\mathfrak g \leftarrow \beta_{j,i}\mathfrak b_j + \beta_{i,i} \mathfrak b_i$\label{hnf:gcd}
      \STATE If $\mathfrak g = 0$ go to step 5.
      \STATE Compute $\gamma \in \mathfrak b_j \mathfrak g^{-1}$ and $\delta \in \mathfrak b_i \mathfrak g^{-1}$ such that $\beta_{j,i}\gamma + \beta_{i,i}\delta = 1$ using Algorithm~\ref{alg:crt2}.
      \STATE $(\mathfrak b_j,\mathfrak b_i) \leftarrow (\mathfrak b_j \mathfrak b_i \mathfrak g^{-1},\mathfrak g)$.
      \STATE $(B_j,B_i) \leftarrow (\beta_{i,i}B_j - \beta_{j,i}B_j,\gamma B_{j} + \delta B_i)$.
      \STATE Normalize $(B_j,\mathfrak b_j)$ and $(B_i,\mathfrak b_i)$ using Algorithm~\ref{alg:normalization}
      \STATE Reduce $B_j$ modulo $\mathfrak d \mathfrak b_j^{-1}$ and $B_i$ modulo $\mathfrak d \mathfrak b_i^{-1}$ using Algorithm~\ref{alg:reduction}.
      \ENDFOR
      \STATE Set $\mathfrak g = \beta_{i,i} \mathfrak b_i + \mathfrak D$. Compute $\gamma \in \mathfrak b_i \mathfrak g^{-1}$ and $\delta \in \mathfrak D \mathfrak g^{-1}$ such that $\gamma \beta_{i,i} + \delta = 1$.
      \STATE Set $B_i \leftarrow \gamma B_i \bmod \mathfrak D \mathfrak g^{-1}$ using Algorithm~\ref{alg:reduction} and $\mathfrak b_i \leftarrow \mathfrak g$, $\beta_{i,i} \leftarrow 1$.
      \STATE $\mathfrak D \leftarrow \mathfrak D \mathfrak g^{-1}$.
    \ENDFOR
    \STATE Move all non-zero rows, together with their coefficient ideals, to the top of $B$.
    \RETURN $(B,(\mathfrak b_i)_i)$.
  \end{algorithmic}
\end{algorithm}

First of all, we want to show that at the beginning of the inner loop at Step~6 the sizes of $B_i,B_j$ and $\mathfrak b_i,\mathfrak b_j$ are bounded. 
We use an inductive argument and begin with the size of the objects at Step~3.
Let $i\in \{1,\dotsc,n\}$ and $j \in \{1,\dotsc,m\}$.
As the ideal $\mathfrak b_i$ is normalized it satisfies 
\[ \min(\mathfrak b_i) \leq \inorm{\mathfrak b_i} \leq \ell^{d^2} \sqrt{\lvert \Delta_K \rvert}.\]
By Proposition~\ref{prop:reduction} the reduction of Step~2 yields
\[ \left\| \beta_{i,j} \right\| \leq d^{3/2} \ell^{d^2} \inorm{\mathfrak d \mathfrak b_i^{-1}}^{1/d} \sqrt{\lvert\Delta_K\rvert} \leq d^{3/2} \ell^{d^2} \min(\mathfrak d) \sqrt{\lvert\Delta_K\rvert}. \]
As $\beta_{i,j}\mathfrak{b_i} \subseteq \mathcal O_K$ the denominator $l \in \Z_{> 0}$ of $\beta_{i,j}$ satisfies $l \leq \min(\mathfrak b_i)$; in particular
\begin{align*} \bs(\beta_{i,j}) &= d \log(\|l \beta_{i,j}\|_\infty) + d\log(l) \\
                                &= 2d \log(l) + d \log(\|\beta_{i,j}\|_\infty) 
                                \in  \OT(\bs(\mathfrak d)/d + \bs(\mathfrak b_i)/d + C ). 
  \end{align*}
  We define $B_\mathrm{id} = d^4 + d^2 \log(\lvert \Delta_K \vert)$ and $B_\mathrm{e} = \bs(\mathfrak d)/d + B_\mathrm{id}/d + C$ respectively.
  The inequalities $B_\mathrm{id} \leq d B_\mathrm{e}$ and $C + d \log(\lvert \Delta_K \rvert) + d^3 \leq B_\mathrm{e}$ will be used throughout the following complexity analysis.

\begin{proposition}\label{prop:size}
  Let $n - m + 1 \leq i \leq n$ and $1 \leq j \leq i - 1$.
  At the beginning of the inner loop at Step 7 the size of the coefficient ideals $\mathfrak b_i$, $\mathfrak b_j$ is bounded by $B_\mathrm{id}$ and the size of the elements of rows $B_i,B_j$ is in $\OT(B_\mathrm{e})$.
\end{proposition}

\begin{proof}
This follows from Steps~10 and 11.
\end{proof}

We will now analyze the complexity of the algorithm.
In order to improve readability we split up the analysis according to the single steps.
Let us first take care of the steps in the loops.

\begin{lemma}\label{prop:compl}
  Let $(A,(\mathfrak a_i)_i)$ be as in the input of Algorithm~\ref{alg:pseudohnf}.
  \begin{enumerate}
    \item
      Steps~6--7 have complexity in $\OT(d^4 B_\mathrm{e})$.
    \item
      Step~8 has complexity in $\OT(d^4 B_\mathrm{e})$.
    \item
      Step~9 has complexity in $\OT(d^2(d+m)B_\mathrm{e})$.
    \item
      Step~10 has complexity in $\OT(d^5 B_\mathrm{e}  + d^3 m B_\mathrm{e})$.
    \item
      Step~11 has complexity in $\OT(d^3 \bs(\mathfrak d) + dm\bs(\mathfrak d) + d^4 m B_\mathrm{e})$.
    \item
      Step~13 has complexity in $\OT(d^2 \bs(\mathfrak d) + d^4 B_\mathrm{e})$.
    \item
      Step~14 has complexity in $\OT(d^3 \bs(\mathfrak d) + dm\bs(\mathfrak d) + d^3 m B_\mathrm{e})$.
  \end{enumerate}
  Thus the inner loop in Steps~6--11 as well as Steps~13--15 is dominated by normalization and reduction yielding an overall complexity in 
  \[ \OT(d^3(d+m)(\bs(\mathfrak d) + d^4 + d^2 \log(\lvert \Delta_K \rvert) + d C)). \]
\end{lemma}

\begin{proof}
  (1): Steps~6--7 are just an application of Algorithm~\ref{alg:crt2} with complexity in $\OT(d^2 B_\mathrm{id} + d^4 B_\mathrm{e} + d^3 C + d^3 \log(\lvert \Delta_K \rvert)) \subseteq \OT(d^4 B_\mathrm{e})$.
  The size of $\gamma$ and $\delta$ is in $\OT(B_\mathrm{id}/d + d B_\mathrm{e} + C) \subseteq \OT(dB_\mathrm{e})$.\par
  (2): The size of $\mathfrak g$ and therefore also the size of $\mathfrak g^{-1}$ is in $\OT(B_\mathrm{id} + d^2 B_\mathrm{e} + dC) \subseteq \OT(d^2 B_\mathrm{e})$.
  As we have already computed $\mathfrak g^{-1}$ in Algorithm~\ref{alg:crt2}, the computation of $\mathfrak b_i \mathfrak b_j \mathfrak g^{-1}$ has complexity in $\OT(d^2 B_\mathrm{id} + d^2(d^2 B_\mathrm{e}) + d^3 C) \subseteq \OT(d^4 B_\mathrm{e})$.
  Note that $\bs(\mathfrak b_i\mathfrak b_j\mathfrak g^{-1}) \in \OT(B_\mathrm{id} + d^2 B_\mathrm{e}) \subseteq \OT(d^2 B_\mathrm{e})$.\par
  (3): Since $\bs(\gamma),\bs(\delta) \in \OT(dB_\mathrm{e})$, computing the scalar vector products has complexity in $\OT(d(d+m)(dB_\mathrm{e}) + dm B_\mathrm{e} + d(d+m) C) \subseteq \OT(d^2(d+m) B_\mathrm{e})$.
  The size of the new elements in row $i$ and $j$ is in $\OT(d B_\mathrm{e})$.\par
  (4): The normalization has complexity in $\OT(d(d^2+m) ( d^2 B_\mathrm{e})+ dm(dB_\mathrm{e}) + d ( d^2 +m ) (\log(\lvert \Delta_K \rvert)+ C))$ which simplifies to $\OT(d^5 B_\mathrm{e} + d^3 m B_\mathrm{e})$.
  While by definition the new ideals have size bounded by $B_\mathrm{id}$, the size of the new elements is in $\OT(d^2 B_\mathrm{e} + dB_\mathrm{e} + d \log(\lvert \Delta_K \rvert) + dC) = \OT(d^2 B_\mathrm{e})$.\par
  (5): Inverting $\mathfrak b_i$ and $\mathfrak b_j$ has complexity in $\OT(dB_\mathrm{id} + d^3 \log(\lvert \Delta_K \rvert) + d^2 C)$ and the multiplication with $\mathfrak d$ is in $\OT(d^2(B_\mathrm{id}+\bs(\mathfrak d)) + d^3 C)$.
  The reduction itself then has complexity in $\OT(d(d^2+m) (B_\mathrm{id} + \bs(\mathfrak d)) + d^2 m (d^2 B_\mathrm{e}) + d^2(d+m)C + d^3m \log(\lvert \Delta_K \rvert))$ which is in $\OT(d^3 \bs(\mathfrak d) + dm \bs(\mathfrak d) + d^4 m B_\mathrm{e})$.\par
  (6): Step~13 is again an application of Algorithm~\ref{alg:crt2} with complexity in $\OT(d^2(B_\mathrm{id} + \bs(\mathfrak d)) + d^4 B_\mathrm{e} + d^3 C + d^3 \log(\lvert \Delta_K \rvert)) \subseteq \OT(d^2 \bs(\mathfrak d) + d^4 B_\mathrm{e})$ and
  again the size of $\gamma$ and $\delta$ is in $\OT(\bs(\mathfrak d)/d + d B_\mathrm{e})$.
  Here we have used that $\bs(\mathfrak D) \leq \bs(\mathfrak d)$ since $\mathfrak D$ is a divisor of $\mathfrak d$.\par
  (7): While the product $\mathfrak D \mathfrak g^{-1}$ was already computed in Algorithm~\ref{alg:crt2}, the computation of $\gamma B_i$ has complexity in $\OT(d(d+m)\bs(\gamma) + dm B_\mathrm{e} + d(d+m) C) = \OT((d+m)\bs(\mathfrak d) + d^2 (d+m)B_\mathrm{e})$.
  Since the entries of $\gamma B_i$ have size in $\OT(\bs(\mathfrak d)/d + dB_\mathrm{e})$ the final reduction is in $\OT(d^3 \bs(\mathfrak d) + dm \bs(\mathfrak d) + d^2 m (\bs(\mathfrak d)/d + d B_\mathrm{e}) + d^2(d+m)C + d^3 m \log(\lvert \Delta_K\rvert))$ which simplifies to $\OT(d^3 \bs(\mathfrak d) + dm\bs(\mathfrak d) + d^3 m B_\mathrm{e})$.
\end{proof}

\begin{theorem}\label{thm:pseudohnf}
  Algorithm~\ref{alg:pseudohnf} is correct and the complexity is in 
  \[ \OT(d^2n(d+m)\max_{i}\bs(\mathfrak a_i) + d^2nm \max_{i,j}\bs(\alpha_{i,j}) + d^3 nm (d+m)(\bs(\mathfrak d) + d^4 + d^2 \log(\lvert \Delta_K \rvert) + d C)).\]
\end{theorem}

\begin{proof}
We first consider the correctness.
Note that the only difference between our algorithm and \cite[Algorithm 3.2]{Cohen1996} are the normalizations.
As the normalizations do not change the module (see Proposition~\ref{prop:normalization}), we conclude that the correctness proof of \cite[Algorithm 3.2]{Cohen1996} carries over.

We now turn to the complexity.
Since the inner loop is executed $O(nm)$ times we conclude using Lemma~\ref{prop:compl} that Steps~5--18 have complexity in $\OT(d^3 nm (d+m)(\bs(\mathfrak d) + d^4 + d^2 \log(\lvert \Delta_K \rvert) + d C))$. 
Now we consider the initialization in Step~1--3.
Denote $\max_{i,j} \bs(\alpha_{i,j})$ and $\max_{i} \bs(\mathfrak a_i)$ by $B_A$ and $B_\mathfrak a$ respectively.
By Proposition~\ref{prop:normalization} Step~1 has complexity in $\OT(d n(d^2+m) B_\mathfrak a + dnm B_A + nd(d^2 +m)(\log(\lvert \Delta_K \rvert) + C))$ and the new elements have size in $\OT(B_\mathfrak a + B_A + d \log(\lvert \Delta_K\rvert) + dC)$.
As in the proof of Lemma~\ref{prop:compl}, computing the product $\mathfrak b_i^{-1}\mathfrak d$ has complexity in $\OT(d B_{\mathrm{id}} + d^3 \log(\lvert \Delta_K \rvert) + d^2(B_{\mathrm{id}}+\bs(\mathfrak d)) + d^3 C)$. Since this is repeated $n$ times this has complexity in $\OT(d^2 n B_{\mathrm{id}} + d^2 n \bs(\mathfrak d) + nd^3C)$.
The reductions then cost $\OT(d(d^2+m)(B_{\mathrm{id}}+\bs(\mathfrak d)) + d^2 m(B_\mathfrak a + B_A +d \log(\lvert \Delta_K\rvert) + dC) +d^2(d+m)C + d^3 m \log(\lvert \Delta_K \rvert))$ per row, i.\,e. $\OT(dn(d^2+m)\bs(\mathfrak d)+ d^2 nm B_A +d^2 nm B_\mathfrak a)$ in total neglecting $C$ and $\log(\lvert \Delta_K \rvert)$.
Now the claim follows.
\end{proof}

\begin{remark}
In \cite{Cohen1996} the following uniqueness result is proven:
Let $M \subseteq \OK^m$ be an $\OK$-module with pseudo-HNF $(A, (\mathfrak a_i)_i)$.
Then any other pseudo-HNF of $M$ has the same coefficient ideals $(\mathfrak a_i)_i$.
For $i,j$ fix representatives $S_{i,j}$ for $K/\mathfrak a_i^{-1}\mathfrak a_j$.
Then using suitable row operations, the pseudo-HNF $(A, (\mathfrak a_i)_i)$ can be transformed into a pseudo-HNF $(B, (\mathfrak a_i)_i)$, $B = (\beta_{i,j})_{i,j}$, of $M$ such that $\beta_{i, j} \in S_{i, j}$ for all $i, j$.
Moreover, $(B, (\mathfrak a_i)_i)$ is unique with this property.
\par
The only difficult step is to find the sets of representatives.
Let $(\alpha_1,\dotsc,\alpha_d)$ be the unique $\Z$-basis of $\mathfrak a_i^{-1}\mathfrak a_j$ obtained from the unique HNF basis of the numerator of $\mathfrak a_i^{-1}\mathfrak a_j$.
Applying Algorithm~\ref{alg:reduction} with $(\alpha_1,\dotsc,\alpha_d)$ instead of an LLL basis we can find for every $\alpha \in K$ an element $\tilde \alpha$ such that $\alpha - \tilde \alpha \in \mathfrak a_i^{-1}\mathfrak a_j$.
Thus the representative $\tilde \alpha$ is defined to be the output of Algorithm~\ref{alg:reduction} applied to $\alpha$.
Note that this yields unique representatives for each class since $\alpha - \beta \in \mathfrak a_i^{-1}\mathfrak a_j$ implies $\tilde \alpha = \tilde \beta$. 
In particular for the representatives $S_{i,j}$ we can just take $\{ \tilde \alpha \mid \alpha \in K \}$.
\end{remark}

\begin{remark}
  The algorithm we described works only in the case that $M \subseteq \mathcal O_K^m$ has rank $m$, and cannot be applied in case the embedding dimension is higher than the rank.
  Moreover there seems to be no obvious adaption of the algorithm to the non-full rank case.
  For an important ingredient of the algorithm is the existence of an integral ideal $\mathfrak m$ of $\mathcal O_K$ with $\mathfrak m \mathcal O_K^m \subseteq M$, allowing us to reduce matrix entries modulo ideals (involving $\mathfrak m$).
  As the existence of such an ideal is equivalent to $M$ being of full rank, any algorithm relying on this modular approach will have this restriction.
\end{remark}

\begin{remark}
  Although using lattice reduction in the normalization step we are able to bound the size of the coefficient ideals, the size already contains a factor $d^4$.
  Together with the expensive ideal operations this explains the high dependency on $d$.
  In addition, the normalization and reduction steps themselves involve a costly lattice reduction algorithm.
  Unfortunately the dependency of the overall complexity of Algorithm~\ref{alg:pseudohnf} on the chosen lattice reduction algorithm is rather involved.
  We find ourselves on the horns of a dilemma -- we have to make sure that the lattice reduction is not too expensive, but at the same time, we need small lattice bases to bound the size of elements and ideals during our algorithm.
\end{remark}
  
\subsection*{Relative versus absolute computations}

We now want to compare the pseudo-HNF algorithm with the HNF algorithm over the integers in situations where we can ``choose'' the structure we work with.
We describe two examples to illustrate the idea.
\par
In practice number fields of large degree are constructed carefully as towers of extensions of type $L \supseteq K \supseteq \Q$ where $K$ is a number field of degree $d$ and $L$ is an extension of $K$ of degree $n$.
The ring of integers $\mathcal O_L$ of $L$ as well as the fractional ideals of $L$ are naturally finitely generated modules of rank $d$ over the Dedekind domain $\mathcal O_K$.
On the other hand, $\mathcal O_L$ as well as the fractional ideals of $L$ are naturally free of rank $dn$ over the principal ideal domain $\Z$.
Thus the computation with ideals in $\mathcal O_L$ can either rely on the pseudo-HNF over $\mathcal O_K$ or on the HNF over $\Z$ and it is not clear which to prefer.
\par
The second situation we have in mind is quite different.
Assume that we are in a situation where we have two finitely generated torsion free $\mathcal O_K$-modules $M$ and $N$ and we are faced with the problem of deciding $M \subseteq N$ and $M = N$.
After imposing further properties on a pseudo-HNF yielding uniqueness the problem can be settled using the pseudo-HNF algorithm.
But as the question only depends on the underlying sets of $M$ and $N$ (discarding the $\mathcal O_K$-structure) the problem can also be sorted out using the HNF over the integers.
Again it is not clear which method to prefer. 
\par
We consider $(A,(\mathfrak a_i)_i)$, a full-rank pseudo-matrix over $\mathcal O_K$ with $A \in K^{n \times n}$ and associated module $M \subseteq \mathcal O_K^n$.
To compute the structure over the integers we have to turn this pseudo-matrix into a $dn\times dn$ matrix over the integers.
As each fractional ideals $\mathfrak a_i$ is isomorphic to $\Z^d$ as a $\Z$-module, we have $M = A_1 \mathfrak a_1 + \dotsb A_n \mathfrak a_n \cong \Z^{dn}$, the isomorphism being induced by the isomorphisms $\mathfrak a_i \to \Z^d$.
Assume that $\beta \in K$ is an element of the $i$-th row of $A$ and $\mathfrak a = \mathfrak a_i$ is the corresponding coefficient ideal of this row.
Denote by $\alpha_1,\dotsc,\alpha_d$ the HNF basis of $\mathfrak a$.
The coefficients of the $d$ products $\beta\alpha_1,\dotsc,\beta\alpha_d$ form a $d\times d$ $\Z$-matrix.
Applying this procedure to all matrix entries of $A$ we obtain a $dn \times dn$ matrix $B$ over the integers, which corresponds to a basis of the free $\Z$-module $M$ of rank $dn$.
These are $n^2$ computations each having complexity in $\OT(d^2 \max(\bs(\alpha_{i,j})) + d \max(\bs(\mathfrak a_i)) + d^2 C)$.
As $\bs(\beta\alpha_i) \leq \bs(\beta ) + \bs(\mathfrak a)/d + C$ the matrix $B$ satisfies $\log(\lvert B \rvert) \leq \max(\bs(\alpha_{ij}))/d + \max(\bs(\mathfrak a_i))/d^2 + C/d$.
Since we know that the matrix $B$ has determinant $\inorm{\mathfrak d}$, where $\mathfrak d$ denotes the determinantal ideal of $(A,(\mathfrak a_i)_i)$, computing the HNF over the integers has complexity in  
\[ \OT((dn)^2 \log( \lvert B \rvert) + (dn)^3 \log( \inorm{\mathfrak d})) \subseteq \OT(dn^2 \max \bs(\alpha_{i,j}) + n^2 \max\bs(\mathfrak a_i) + d^2 n^3 \bs(\mathfrak d) + d^2n C) .\]
Combining this with the complexity of computing $B$ we get an overall complexity in
\[ \OT(d^2 n^2 \max \bs(\alpha_{i,j}) + d n^2 \max \bs(\mathfrak a_i) + d^2 n^3  \bs(\mathfrak d) + d^3 n^2 C). \]
While the dependency on $n$ is the same as in the pseudo-HNF case (see Theorem~\ref{thm:pseudohnf}), the powers of $d$ are slightly lower due to the absence of ideal arithmetic involving normalization and reduction.
We conclude that one should always use the HNF over the rational integers if possible.
Note that this discussion depends on the chosen pseudo-HNF algorithm and not on the notion of the pseudo-HNF itself and of course it is possible that more sophisticated approaches yield different conclusions.

%%%%%%%%%%%%%%%%%%%%%%%%%%%%%%%%%%%%%%%%%%%%%%%%%%%%%%%%%%%%%%%%%%%%%%%%%%%%%%%%
%
%  The pseudo-SNF algorithm
%
%%%%%%%%%%%%%%%%%%%%%%%%%%%%%%%%%%%%%%%%%%%%%%%%%%%%%%%%%%%%%%%%%%%%%%%%%%%%%%%%

\section{The pseudo-SNF algorithm}

The notion of pseudo-Smith normal form (pseudo-SNF) was introduced by
Cohen~\cite{Cohen1996}, see also~\cite[Sec. 1.7]{Cohen2000}, to describe
quotients of $\OK$-modules and to generalize the Smith normal form to modules
over Dedekind domains.  For simplicity, we restrict ourselves to quotients of
modules of the same rank, but these results can easily be generalized to
quotients of the form 
$M/N$ where $\operatorname{rank}(M) > \operatorname{rank}(N)$.  Let $A\in
K^{n\times n}$ be a non-singular matrix and $I = (\bg_1,\dotsc , \bg_n)$, $J =
(\ag_1,\dotsc,\ag_n)$ families of fractional ideals. We say that $(A,I,J)$ is
an integral bi-pseudo matrix for the rank-$n$ $\OK$-modules 
\begin{align*}
 M &= \bg_1\eta_1\oplus\cdots \oplus\bg_n\eta_n\\
 N &= \ag_1\omega_1\oplus\cdots \oplus\ag_n\omega_n
\end{align*}
if $a_{i,j}\in\bg_i\ag_j^{-1}$ for all $1 \leq i,j \leq n$, and the linear transformation $f:K^n\rightarrow K^n$ associated to $A$ satisfies
\[ f(\omega_j) = a_{1,j}\eta_1+\cdots + a_{n,j}\eta_n \]
for all $1 \leq j \leq n$.
Then, it is shown in~\cite[Sec. 1.7]{Cohen2000} that the 
quotient $Q$ associated to $(A,I,J)$ is 
$$Q = (\bg_1 \eta_1 \oplus \cdots \oplus \bg_n \eta_n)/(\ag_1 f(\omega_1)\oplus \cdots \oplus \ag_n f(\omega_n)).$$ 
In the 
case of two modules $M,N$ of rank $n$ with $N\subset M$, the elementary divisor theorem ensures the 
existence of $(\alpha_1,\dotsc,\alpha_n)\in K^n$, fractional ideals $\bg_1,\dotsc,\bg_n$ and 
integral ideals $(\dg_1,\dotsc,\dg_n)$ with $\dg_{i-1}\subset\dg_i$ such that 
\begin{equation}\label{eq:el_div}
M = \bg_1\alpha_1\oplus \cdots \oplus \bg_n\alpha_n\ \text{and}\ N = 
\dg_1\bg_1\alpha_1\oplus \cdots \oplus \dg_n\bg_n.
\end{equation}

In the language of bi-pseudo matrices the elementary divisor theorem takes the following form.

\begin{theorem}[Definition of the pseudo-SNF]\label{th:SNF}
Let $(A,I,J)$ be an integral bi-pseudo matrix with $I=(\bg_i)_{i\leq n}$ and 
$J = (\ag_i)_{i\leq n}$. There exist ideals $(\bg_i')_{i\leq n}$, $(\ag_i')_{i\leq n}$ and 
$n\times n$ matrices $U,V$ such that with $\dg_i = \ag_i'\bg_i'^{-1}$ we have 
\begin{enumerate}
 \item $\prod_i\ag_i = \det(U)\prod_i\ag_i'$ and $\prod_i\bg_i' = \det(V)\prod_i\bg_i$.
 \item $VAU$ is the $n\times n$ identity matrix.
 \item The $\dg_i$ are integral and $\dg_{i-1}\subset \dg_i$ for $2\leq n\leq n$. 
 \item For all $i,j$, $u_{i,j}\in \ag_i\ag_j'^{-1}$ and $v_{i,j}\in \bg_i'\bg_j^{-1}$. 
\end{enumerate}
In this case, the triplet $(UAV,(\bg_i')_{i\leq n},(\ag_i')_{i\leq n}$ is called a pseudo-SNF of 
$(A,I,J)$.
\end{theorem}

Our algorithm for computing the pseudo-SNF of an integral bi-pseudo matrix is derived from~\cite[Alg. 1.7.3]{Cohen2000}. 
The possibility of working modulo the determinantal ideal was considered (although not explicitly described), 
but as for the computation of the pseudo-HNF, this does not prevent the growth of the denominators. We 
propose a modular version of~\cite[Alg. 1.7.3]{Cohen2000} incorporating the normalization. We restrict 
ourselves to the case of non-singular square integral pseudo-matrices, but this result can be extended 
to the rectangular case easily.

\begin{algorithm}[ht]
  \caption{pseudo-SNF of a full-rank square bi-pseudo matrix}
  \begin{algorithmic}[1]\label{alg:pseudosnf}
 \REQUIRE $(A,(\bg_k)_{k\leq n},(\ag_k)_{k\leq n})$ integral $n\times n$ bi-pseudo matrix and 
$\dg = \det(A)\prod_i \ag_i\bg_i^{-1}$
\STATE Compute $(\bg_i^{-1})_{i\leq n}$. 
\STATE Normalize $(A'_j,\ag_j)$ and $(A_i,\bg_i^{-1})$ for $i,j\leq n$.
\STATE Reduce $a_{i,j}\bmod \ag_i^{-1}\bg_j$ for $i,j\leq n$.
\FOR{$i = n , \dotsc , 1$}
\STATE $\operatorname{StepOver}\leftarrow\operatorname{false}$
\WHILE {$\operatorname{StepOver} = \operatorname{false}$}
\STATE $M\leftarrow\operatorname{ColPivot}(M,\dg,i)$ with Algorithm~\ref{alg:colpivot}.
\STATE $M,\operatorname{StepOver}\leftarrow\operatorname{RowPivot}(M,\dg,i)$ with Algorithm~\ref{alg:rowpivot}.
\IF{$\operatorname{StepOver}=\operatorname{true}$}
\STATE $\bg\leftarrow \ag_i\bg_i$.
\FOR{$1\leq k,l<i,\operatorname{StepOver}=\operatorname{true}$}
\IF{$a_{k,l}\ag_l\bg_k^{-1}\not\subset\bg$}
\STATE Let $g\in\g:=\bg_i\bg_k^{-1}$ such that $a_{k,l}g\notin\ag_i\ag_l^{-1}$.
\STATE $A_i\leftarrow A_i + gA_k$, $\operatorname{StepOver}\leftarrow\operatorname{false}$.
\STATE Reduce $a_{k,i}\bmod \dg\ag_k^{-1}\bg_i$ for $k\leq i$ using Algorithm~\ref{alg:reduction}.
\ENDIF
\ENDFOR
\ENDIF
\IF{$\operatorname{StepOver}=\operatorname{true}$}
\STATE $\ag_i\leftarrow a_{i,i}\ag_i$, $a_{i,i}\leftarrow 1$. 
\STATE $\dg_i\leftarrow \gcd(\ag_i\bg_i^{-1},\dg)$.
\STATE $\dg\leftarrow \dg\dg_i^{-1}$. 
\ENDIF
\ENDWHILE
\ENDFOR
\RETURN $\dg_1,\dotsc,\dg_n$.
  \end{algorithmic}
\end{algorithm}

For the sake of clarity, we give the description of row and column operations in 
separate algorithms. These are very similar to the row operations for the pseudo-HNF
computation. The main difference is that we cannot use normalization on the pivots since 
we need a strictly increasing chain of ideals to ensure that the algorithm terminates. 

\begin{algorithm}[ht]
  \caption{ColPivot}
  \begin{algorithmic}[1]\label{alg:colpivot}
\REQUIRE $(A,(\bg_k)_{k\leq n},(\ag_k)_{k\leq n})$, $i\leq n$, $\dg$
\FOR {$j = i-1,\dotsc,1$}
\IF{$a_{i,j}\neq 0$}
\STATE $\g\leftarrow a_{i,i}\ag_i + a_{i,j}\ag_j$
\STATE Compute $\gamma \in \ag_i \g^{-1}$ and $\delta \in \mathfrak a_j \g^{-1}$ such that $a_{i,i}\gamma + a_{i,j}\delta = 1$ using Algorithm~\ref{alg:crt2}.
\STATE $(A'_j,A'_i)\leftarrow (a_{i,j}A'_j-a_{i,i}'A_i,\gamma A'_i + \delta A'_j)$.
\STATE $(\ag_j,\ag_i)\leftarrow (\ag_i\ag_j\g^{-1},\g)$. 
\STATE Normalize $(A'_j,\mathfrak a_j)$ and $(A'_i,\mathfrak a_i)$ using Algorithm~\ref{alg:normalization}.
\STATE Reduce $a_{j,k}\bmod \dg\ag_j^{-1}\bg_k$ and $a_{i,k}\bmod \dg\ag_i^{-1}\bg_k$ for $k\leq i$ using Algorithm~\ref{alg:reduction}.
\ENDIF
\ENDFOR
  \end{algorithmic}
\end{algorithm}

\begin{algorithm}[ht]
  \caption{RowPivot}
  \begin{algorithmic}[1]\label{alg:rowpivot}
\REQUIRE $(A,(\bg_k)_{k\leq n},(\ag_k)_{k\leq n})$, $i\leq n$, $\dg$
\STATE $\operatorname{StepOver}\leftarrow\operatorname{true}$
\FOR {$j = i-1,\dotsc,1$}
\IF{$a_{j,i}\neq 0$}
\STATE $\g\leftarrow \bg_i^{-1} + a_{j,i}\bg_j^{-1}$
\STATE Compute $\gamma \in \bg_i^{-1} \g^{-1}$ and $\delta \in \mathfrak b_j \g^{-1}$ such that $\gamma + a_{j,i}\delta = 1$ using Algorithm~\ref{alg:crt2}.
\STATE $(A_j,A_i)\leftarrow (a_{j,i}A_j-a_{i,i}A_i,\gamma A_i + \delta A_j)$.
\STATE $(\bg_j,\bg_i)\leftarrow (\bg_i\bg_j\g,\g^{-1})$. 
\STATE Normalize $(A_j,\mathfrak b_j^{-1})$ and $(A_i,\mathfrak b_i^{-1})$ using Algorithm~\ref{alg:normalization}.
\STATE Reduce $a_{k,j}\bmod \dg\ag_k^{-1}\bg_j$ and $a_{k,i}\bmod \dg\ag_k^{-1}\bg_i$ for $k\leq i$ using Algorithm~\ref{alg:reduction}.
\STATE $\operatorname{StepOver}\leftarrow\operatorname{false}$.
\ENDIF
\ENDFOR
\RETURN $\operatorname{StepOver}$.
  \end{algorithmic}
\end{algorithm}

Before proving the correctness of Algorithm~\ref{alg:pseudosnf}, let us prove that 
the elementary operations performed on the bi-pseudo matrix do not change the quotient 
module that it represents.

\begin{lemma}\label{lem:quotient_normalization}
Let $A,(\bg_i)_{i\leq n},(\ag_i)_{i\leq n}$ be a bi-pseudo matrix. 
Then the operations 
\begin{itemize}
 \item $\bg_i\leftarrow (\alpha)\bg_i$, $A_i\leftarrow (\alpha)A_i$,
 \item $\ag_j\leftarrow (\alpha)\ag_j$, $A'_j\leftarrow (1/\alpha)A'_j$,
\end{itemize}
do not modify the quotient module. In particular, the operations 
\begin{itemize}
 \item Normalize $(A_i,\mathfrak b_i^{-1})$ using Algorithm~\ref{alg:normalization}.
 \item Normalize $(A'_j,\mathfrak a_j)$ using Algorithm~\ref{alg:normalization}.
\end{itemize}
do not modify the quotient module. Moreover, they leave the integral ideal 
$\sum_{i,j}a_{i,j}\ag_j\bg^{-1}_i$ unchanged. 
\end{lemma}

\begin{proof}
We keep the same notation as before: $M=\bigoplus_i\bg_i\eta_i$, $N=\bigoplus_i\ag_i\omega_i$. Then 
if we perform $\bg_i\leftarrow (\alpha)\bg_i$, we need to update $\eta_i\leftarrow \frac{1}{\alpha}\eta_i$ 
to preserve $M$. Then for all $1\leq j \leq n$, 
\begin{align*}
f(\omega_j) &= a_{1,j}\eta_1 + \cdots + a_{n,j}\eta_n\\
&= a_{1,j}\eta_1 + \cdots a_{i,j}\alpha \left(\frac{\eta_i}{\alpha}\right) + \cdots + a_{n,j}\eta_n.
\end{align*}
Thus the $i$-th row of $A$ gets multiplied by $\alpha$. 

Likewise, if $\ag_j\leftarrow (\alpha)\ag_j$, we need to update $\omega_j\leftarrow \frac{1}{\alpha}\omega_j$ 
to preserve $N$, which means by linearity that $f(\omega_j) = \frac{1}{\alpha}f(\omega_j)$. This means that 
the $j$-th column of $A$ gets multiplied by $\frac{1}{\alpha}$. 
\end{proof}

\begin{proposition}
Algorithm~\ref{alg:pseudosnf} terminates and gives the pseudo-SNF of the input.
\end{proposition}

\begin{proof}
The proof of the termination of the non-modular version of Algorithm~\ref{alg:pseudosnf} is 
given in the proof of~\cite[Alg. 1.7.3]{Cohen2000}, while the correctness of the modular 
version is treated in the integer case in the proof of~\cite[Alg. 2.4.14]{Cohen1996}. We 
only highlight the points where our context could induce a difference. The main argument 
showing that Steps~3 to~19 will only be repeated a finite amount of times is that the 
integral ideal $a_{i,i}\ag_i\bg_i^{-1}$ at step $i$ either increases or is left unchanged 
therefore triggering the end of the loop. The main difference with~\cite[Alg. 1.7.3]{Cohen2000} is 
that we normalize and reduce $\ag_j,A_j$ for $j\leq i$ to prevent coefficient explosion. Without 
taking into account modular reductions, Steps~5 and~6 transform the triplet $(a_{i,i},\ag_i,\bg_i^{-1})$ 
into 
$$\left(\frac{1}{\alpha_i\beta_i},\sum_{j\leq i}a_{i,j}(\alpha_j)\ag_j,\beta_i\bg_i^{-1}+\sum_{j<i}a'_{j,i}
(\beta_j)\bg_j^{-1}\right),$$
where the $a'_{i,j}$ are the entries of $A$ after Step~5, $\alpha_j$ are the minima used to 
normalize $\ag_j,A'_j$ and $\beta_j$ are the minima used to normalize $\bg^{-1}_i,A_i$. As 
in~\cite[Alg. 1.7.3]{Cohen2000}, $a_{i,i}\ag_i\bg_i^{-1}\subseteq\sum_{i,j}a_{i,j}\ag_j\bg^{-1}_i \subseteq\OK$ is integral 
since according to Lemma~\ref{lem:quotient_normalization} the normalization steps do not alter 
this property. In addition, we see that $a_{i,i}\ag_i\bg_i^{-1}$ can only increase. 
To show that the modular reductions do not 
prevent the algorithm to terminate, we compare the evolution of $a_{i,i}\ag_i\bg_i^{-1}$ and 
$\overline{a_{i,i}}\ag_i\bg_i^{-1}$ where the $a_{i,j}$ are the coefficients of the matrix $A$ 
during the course of Algorithm~\ref{alg:pseudosnf} executed without the modular reductions while 
the $\overline{a_{i,j}}$ are the same values when Algorithm~\ref{alg:pseudosnf} is run with 
modular reductions. The analysis of~\cite[Alg. 1.7.3]{Cohen2000} still holds for the 
non-modular version of Algorithm~\ref{alg:pseudosnf} which only differs from ~\cite[Alg. 1.7.3]{Cohen2000} 
by the normalizations. 
The essential argument for the termination of Algorithm~\ref{alg:pseudosnf} is that we have 
\begin{align*}
a_{i,i}\ag_i\bg_i^{-1}&\subseteq \overline{a_{i,i}}\ag_i\bg_i^{-1} + \dg\subseteq\OK. %\\
\end{align*}
Indeed, 
let $\overline{a_{i,j}}:=a_{i,j}\bmod \dg\ag_j^{-1}\bg_i$, 
and $d_{i,j}:= \overline{a_{i,j}}-a_{i,j}\in\dg\ag_j^{-1}\bg_i$, then for each $x\in\ag_i\bg_i^{-1}$, 
\begin{align*}
a_{i,i}x &= \underbrace{\overline{a_{i,j}}x}_{\in \overline{a_{i,j}}\ag_i\bg_i^{-1}} + \underbrace{d_{i,i}x}_{\in\dg},
\in \overline{a_{i,i}}\ag_i\bg_i^{-1} + \dg\\
\overline{a_{i,i}}x &= \underbrace{a_{i,i}x}_{\in\OK} - \underbrace{d_{i,i}x}_{\in\OK}\in\OK.
\end{align*}
Therefore, $\overline{a_{i,i}}\ag_i\bg_i^{-1} + \dg$ is an integral ideal which strictly increases and can only 
stabilize when the termination condition is reached.

The other main claim to be verified is the correctness of the 
modular approach. We adapt and reuse the argument presented in the proof of~\cite[Alg. 2.4.14]{Cohen1996}. 
We extend the notion of an $i\times i$ submatrix to pseudo-matrices by taking into account the 
coefficient ideals. Then let $\delta_i(A,I,J)$ be the sum of the determinantal ideal of all the $i\times i$ 
submatrices of $(A,I,J)$. As in the integer case, this value is an integral ideal invariant under the 
transformations performed in the non-modular version of Algorithm~\ref{alg:pseudosnf}. 
In addition, we need to prove that 
the modular reductions of the form $a_{i,j}\leftarrow a_{i,j} \bmod \dg\ag_j^{-1}\bg_i$ for 
$\dg\subseteq \OK$ performed on 
rows and columns preserve $\gcd\left(\det(A)\prod_i\ag_i\bg_i^{-1},\dg\right)$. 
From the symmetry between row and 
column reduction, it suffices to prove this for row reductions. Our determinantal 
ideal is given by 
$$\det(A)\prod_i\ag_i\bg_i^{-1} = \left( \sum_{\sigma\in S_n}\prod_i a_{i,\sigma(i)}\right) \prod_i\ag_i\bg_i^{-1}.$$
Let $\overline{a}_{i,j} := a_{i,j}\bmod \dg\ag_i^{-1}\bg_j$ and $d_{i,j}\in \dg\ag_i^{-1}\bg_j$ such that 
$\overline{a}_{i,j} = d_{i,j} + a_{i,j}$. In particular, for $\sigma\in S_n$, we have 
\begin{align*}
a_{1,\sigma(1)}\cdots &\overline{a}_{i,\sigma(i)}\cdots a_{n,\sigma(n)}\prod_j \ag_j\bg_j^{-1} \\
&=a_{1,\sigma(1)}\cdots (d_{i,\sigma(i)} + a_{i,\sigma(i)})\cdots a_{n,\sigma(n)}\prod_j \ag_i\bg_{\sigma(j)}^{-1}\\
&\subseteq \prod_j a_{j,\sigma(j)}\ag_j\bg_{\sigma(j)}^{-1} + d_{i,\sigma(i)}\ag_i\bg_{\sigma(i)}^{-1}
\underbrace{\prod_{j\neq i}a_{j,\sigma(j)}\ag_j\bg_{\sigma(j)}^{-1}}_{\in\OK}\\
&\subseteq \prod_j a_{j,\sigma(j)}\ag_j\bg_{\sigma(j)}^{-1} + \dg.
\end{align*}
This means that $\det(\overline{A})\prod_i \ag_i\bg_i^{-1} \subseteq \det(A)\prod_i \ag_i\bg_i^{-1} + \dg$ 
where $\overline{A} = (\overline{a}_{i,j})_{i,j\leq n}$. We show by using the same argument that 
$\det(A)\prod_i \ag_i\bg_i^{-1} \subseteq  \det(\overline{A})\prod_i \ag_i\bg_i^{-1} + \dg$, thus concluding 
that $\det(\overline{A})\prod_i \ag_i\bg_i^{-1} + \dg = \det(A)\prod_i \ag_i\bg_i^{-1} + \dg$.
Let the $(\dg_i)_{i\leq n}$ be the 
elementary divisors of the quotient module represented by $(A,(\bg_i)_{i\leq n},(\ag_i)_{i\leq n})$, 
and $\dg:= \prod_i\dg_i$. Let $S = (I_n,(\bg'_i)_{i\leq n},(\ag'_i)_{i\leq n})$, 
be the bi-pseudo matrix resulting from the manipulation described in Algorithm~\ref{alg:pseudosnf} on 
the input $(A,(\bg_i)_{i\leq n},(\ag_i)_{i\leq n})$, and $\Gamma$ the actual pseudo-SNF of $M$. Then, as in the proof of~\cite[Alg. 2.4.14]{Cohen1996}, 
we have 
\begin{align*}
\dg_i\cdots\dg_n &= \gcd(\dg , \delta_{n-i+1}(\Gamma))\\
&= \gcd(\dg,\delta_{n-i+1}(A,(\bg_i)_{i\leq n},(\ag_i)_{i\leq n}))\\
&= \gcd(\dg,\delta_{n-i+1}(S))\\
& = \gcd(\dg,\ag'_i\left.\bg'_i\right.^{-1}\cdots \ag'_n\left.\bg'_n\right.^{-1}).
\end{align*}
After setting $\mathfrak{P}_i = \dg_{i+1}\dotsb \dg_n$ we have
$$\gcd\left(\dg\mathfrak{P}_i^{-1},\left(\ag'_{i+1}\left.\bg'_{i+1}\right.^{-1}\cdots \ag'_n\left.\bg'_n\right.^{-1}
  \right)\mathfrak{P}_i^{-1}\right) =\OK,$$ 
and 
$$\gcd\left(\dg \mathfrak{P}_i^{-1},\left(\ag'_{i}\left.\bg'_{i}\right.^{-1}\cdots \ag'_n\left.\bg'_n\right.^{-1}
  \right)\mathfrak{P}_i^{-1}\right) =\dg_i.$$ 
Therefore, $\dg_i = \gcd(\dg\mathfrak{P}_i^{-1},\ag'_{i}\left.\bg'_{i}\right.^{-1})$, which shows by induction 
the correctness of Algorithm~\ref{alg:pseudosnf}. 
\end{proof}

To analyze the cost of the pseudo-SNF computation, we first consider the blocks ColPivot and RowPivot which resemble 
the row operations performed during the pseudo-HNF computation. In the following, we keep the notations 
$\Bid$ and $\Be$ from the analysis of the pseudo-HNF algorithm. 

\begin{proposition}\label{prop:colpivot}
The cost of Algorithm~\ref{alg:rowpivot} and Algorithm~\ref{alg:colpivot} is in 
$$\tilde{O}\left( d^3(d+n)\left( \bs(\dg) + d^4 + d^2\log(|\Delta_K|) + dC\right) \right).$$
\end{proposition}

\begin{proof}
This is derived almost entirely from Lemma~\ref{prop:compl} which was used in the 
analysis of the pseudo-HNF algorithm. The only notable difference, which does not 
impact the complexity, occurs in the modular reduction (Step~8 of Algorithm~\ref{alg:colpivot} 
and Step~9 of Algorithm~\ref{alg:rowpivot}). First, the reduction is no longer modulo 
$\dg\bg_i^{-1}$ but modulo $\dg\ag_j^{-1}\bg_k$. However, the size of these ideals remains 
in $\tilde{O}\left(\bs(\dg) + \Bid\right)$ thus not impacting the analysis. Also, 
the reduction of the entries of a row or a column is modulo a different ideal for 
each entry, thus preventing us from reusing the reduced basis. However, considering 
the bounds on the size of the elements in play, this does not change the complexity. 
\end{proof}

\begin{proposition}\label{prop:verif_snf}
The cost of Steps~10 to~18 of Algorithm~\ref{alg:pseudosnf} is in 
$$\tilde{O}\left( nd^2(d+n)\left( \bs(\dg) + d^4 + d^2\log(|\Delta_K|) \right) + nd^3C \right).$$
\end{proposition}

\begin{proof}
  Computing $\bg$ costs $\tilde{O}(d^2\Bid + d^3\log(|\Delta_K|) + d^3C)$. It requires inverting 
$\bg_i^{-1}$ and multiplying it with $\ag_i$. This is done only once. 

Calculating $\ag_l\bg_k^{-1}$ does not involve inversion (we have $\bg_k^{-1}$) and 
thus costs $\tilde{O}(d^2\Bid + d^3C)$. 
Then, calculating $a_{k,l}\ag_l\bg_k^{-1}$ costs $\tilde{O}(d^3\Be)$. Checking 
whether $a_{k,l}\ag_l\bg_k^{-1}\subset\bg$ can be done by calculating the pseudo-HNF of 
$(H_1^t | H_2^t )^t$ where $H_1$ is the $\Z$-basis matrix of $\bg$ and 
$H_2$ is that of $a_{k,l}\ag_l\bg_k^{-1}$. If it is the same as $H_1$, it means that 
$a_{k,l}\ag_l\bg_k\subset\bg$. The entries of the matrix representing the $\Z$-basis 
of $\bg$ have their size in $\tilde{O}\left( {\Bid}/{d^2}\right)$, while the 
size of the entries of the matrix of the $\Z$-basis of $a_{k,l}\ag_l\bg_k^{-1}$ are 
in $\tilde{O}\left( {\Bid}/{d^2} + \Be + {C}/{d}\right)$. Therefore 
computing the HNF of their concatenation costs 
$\tilde{O}\left( d\Bid + d^3\Be + d^2C\right) \subseteq \tilde{O}(d^3\Be)$. So the 
search for $k,l$ such that $a_{k,l}\ag_l\bg_k^{-1}\not\subset\bg$ costs 
$\tilde{O}(n^2d^3\Be)$. 

Once we have $k,l$, we find $g$ by checking if $a_{k,l}\alpha_h\in\ag_i\ag_l^{-1}$ where the 
$(\alpha_h)_{h\leq d}$ are the elements of the $\Z$-basis of $\g$. Calculating $\g$ and $\ag_i\ag_l^{-1}$ 
has the same 
cost as calculating $\bg$, that is $\tilde{O}(d^2\Bid + d^3\log(|\Delta_K|) + d^3C)$. As 
$\bs(a_{k,l}\alpha_h)\in\tilde{O}\left( \Be + {\Bid}/{d} + C \right)$, the entries of the 
corresponding vector are in $\tilde{O}\left( {\Be}/{d} + {\Bid}/{d^2} + {C}/{d}\right)$. 
Likewise, the entries of the matrix of the $\Z$-basis of $\ag_i\ag_l^{-1}$ are in 
$\tilde{O}\left( {\Bid}/{d^2}\right)$. Therefore, solving the linear system to verify if 
$a_{k,l}\alpha_h\in\ag_i\ag_l^{-1}$ costs $\tilde{O}(d^2\Be)$, and this is repeated at 
most $d$ times, at a total cost of $\tilde{O}(d^3\Be)$. The resulting element $g\in\g$ satisfies 
$\bs(g)\in\tilde{O}\left( {\Bid}/{d}\right)$. 

The step $A_i\leftarrow A_i + gA_k$ costs $\tilde{O}\left( d(d+n){\Bid}/{d} + dn\Be + d(n+d)C\right)$, 
and the resulting entries of $a_{k,i}$ of $A_i$ satisfy 
$\bs(a_{k,i})\in\tilde{O}\left( {\Bid}/{d} + \Be + C\right) \subseteq \tilde{O}(\Be)$. 

Finally, as $\bs(\dg\ag_k^{-1}\bg_i)\in\tilde{O}\left( \bs(\dg) + \Bid\right)$, the cost of the 
$n$ reduction of $a_{k,i}$ modulo $\dg\ag_k^{-1}\bg_i$ is in 
$$\tilde{O}\left( n \left( d^3(\bs(\dg) + \Bid) + d^2\Be + d^3\log(|\Delta_K|) + d^3C\right) \right).$$
This and the term in $\tilde{O}(n^2d^3\Be)$ are the two dominant steps. The result follows by substituting 
the values of $\Be$ and $\Bid$ by their expression in terms of the invariants of the field and $\bs(\dg)$. 
\end{proof}

\begin{proposition}\label{prop:pseudosnf}
Let $B_A = \max_{i,j}\bs(a_{i,j})$ and $B_\ag = \max_{i,j}(\bs(\ag_i),\bs(\bg_j))$. 
The cost of Algorithm~\ref{alg:pseudosnf} is in 
$$\tilde{O}\left( nd((d+n)^2\bs(\dg) + nd^2)(d^4+d^2\log(|\Delta_K|)+\bs(\dg)) + n^2d^2(\bs(\dg)C + B_A + B_\ag) \right).$$
\end{proposition}

\begin{proof}
First, let us estimate the cost of Steps~1 to~3. Inverting the $\bg_i$ costs 
$\tilde{O}\left( n \left( dB_\ag + d^3\log(|\Delta_K|) + d^2C\right) \right)$. As in the proof 
of the complexity of Algorithm~\ref{alg:pseudohnf}, the normalization costs 
$\tilde{O}\left( dn(d^2 + nB_\ag) + dn^2B_A + nd(d^2+n)(\log(|\Delta_K|) + C)\right)$. The new 
elements have size $\tilde{O}\left( B_\ag + B_A + d\log(|\Delta_K|) + dC\right)$. Then, as we 
already have the $\bg_i^{-1}$, calculating the $\dg\ag_j\bg_i^{-1}$ costs 
$\tilde{O}\left( n^2\left( d^2(\Bid + \bs(\dg)) + d^3 C\right) \right)$. Finally, the 
cost of the subsequent reduction is in 
$$\tilde{O}\left( n^2d^2\left( d(d^4 + d^2\log(|\Delta|) + \bs(\dg)) + B_A + B_\ag + dC\right) \right),$$
which is the dominant step of this precalculation.

Now, let us analyze the main loop of the algorithm. 
The condition $\operatorname{StepOver} = \operatorname{true}$ is attained in 
at most $\tilde{O}(\bs(a_{i,i}\ag_i\bg_i^{-1}))$ since $a_{i,i}\ag_i\bg_i^{-1}$ becomes 
strictly larger at each iteration. Therefore, the number of iterations is in 
$\tilde{O}(\log(\inorm{\dg}))$. So Algorithm~\ref{alg:rowpivot}, Algorithm~\ref{alg:colpivot} and 
the Steps~10 to~18 are 
executed $\tilde{O}(n\bs(\dg)/d)$ times. We obtain the cost of the main loop by adding 
the estimated costs found in 
Proposition~\ref{prop:colpivot} and Proposition~\ref{prop:verif_snf} and multiplying this by 
$n\bs(\dg)/d$.
\end{proof}

\bibliographystyle{elsarticle-harv}
\bibliography{paper}

\end{document}